\documentclass[onefignum,onetabnum]{siamonline220329}

\usepackage{lipsum}
\usepackage{amsfonts}
\usepackage{graphicx}
\usepackage{epstopdf}
\usepackage{algorithm}
\usepackage{algpseudocode}
\ifpdf
  \DeclareGraphicsExtensions{.eps,.pdf,.png,.jpg}
\else
  \DeclareGraphicsExtensions{.eps}
\fi

\newsiamremark{remark}{Remark}
\newsiamremark{hypothesis}{Hypothesis}
\crefname{hypothesis}{Hypothesis}{Hypotheses}
\newsiamthm{claim}{Claim}
\newsiamthm{problem}{Problem}
\crefname{problem}{Problem}{Problem}

\headers{Quantifying the structural stability of simplicial homology}{N. Guglielmi, A. Savostianov, F. Tudisco}

\title{Quantifying the structural stability of simplicial homology\thanks{Submitted to the editors on DATE.}}

\author{Nicola Guglielmi\thanks{Gran Sasso Science Institute, L'Aquila, Italy 
  (\email{nicola.guglielmi@gssi.it}, \email{anton.savostianov@gssi.it}, \email{francesco.tudisco@gssi.it}).}
\and Anton Savostianov\footnotemark[2]
\and Francesco Tudisco\footnotemark[2]}

\usepackage{amsopn}
\DeclareMathOperator{\diag}{diag}

\makeatletter
\newcommand*{\addFileDependency}[1]{
  \typeout{(#1)}
  \@addtofilelist{#1}
  \IfFileExists{#1}{}{\typeout{No file #1.}}
}
\makeatother

\newcommand*{\myexternaldocument}[1]{%
    \externaldocument{#1}%
    \addFileDependency{#1.tex}%
    \addFileDependency{#1.aux}%
}

\ifpdf
\hypersetup{
  pdftitle={An Example Article},
  pdfauthor={D. Doe, P. T. Frank, and J. E. Smith}
}
\fi

\myexternaldocument{ex_supplement}

\usepackage{enumitem}
\usepackage{faktor}
\usepackage{amsmath}
\usepackage{amssymb}
\usepackage{comment}

\def\tilde#1{\widetilde{#1}}
\def\bar#1{\overline{#1}}
\def\mc#1{\mathcal{#1}}

\newcommand\La{ L_0}

\DeclareMathOperator*{\im}{\mathrm{im}}

\newcommand\Lap{{L}_1}

\DeclareMathOperator*{\Sym}{\mathrm{Sym}}
\DeclareMathOperator*{\vect}{\mathrm{diagvec}}
\newcommand\eps{\varepsilon}

\renewcommand\vec[1]{\boldsymbol{#1}}

\newcommand\scal[1]{\left\langle {#1} \right\rangle}

\makeatletter
\DeclareRobustCommand*{\mfaktor}[3][]
{
   { \mathpalette{\mfaktor@impl@}{{#1}{#2}{#3}} }
}
\newcommand*{\mfaktor@impl@}[2]{\mfaktor@impl#1#2}
\newcommand*{\mfaktor@impl}[4]{
   \settoheight{\faktor@zaehlerhoehe}{\ensuremath{#1#2{#3}}}%
   \settoheight{\faktor@nennerhoehe}{\ensuremath{#1#2{#4}}}%
      \raisebox{-0.5\faktor@zaehlerhoehe}{\ensuremath{#1#2{#3}}}%
      \mkern-4mu\diagdown\mkern-5mu%
      \raisebox{0.5\faktor@nennerhoehe}{\ensuremath{#1#2{#4}}}%
}

\newsiamthm{example}{Example}
\crefname{example}{Example}{Example}

\makeatletter
\newcommand*\bigcdot{\mathpalette\bigcdot@{.5}}
\newcommand*\bigcdot@[2]{\mathbin{\vcenter{\hbox{\scalebox{#2}{$\m@th#1\bullet$}}}}}
\makeatother

\usepackage{stackengine}

\newcommand\obullet[1]{\ensurestackMath{\stackon[0pt]{#1}{\mkern1mu\bigcdot}}}

\renewcommand{\dot}{\obullet}

\makeatother

\usepackage{ wasysym }

\usepackage{tikz}
\usetikzlibrary{arrows.meta}
\usetikzlibrary{patterns}
\usetikzlibrary{matrix}

\newcommand{\AxisRotator}[1][rotate=0]{%
    \tikz [x=0.15cm,y=0.15cm,line width=.2ex,-stealth,#1] \draw (0,0) arc (-150:150:1 and 1);%
}
\newcommand{\AxisRotatorMirror}[1][rotate=0]{%
    \tikz [x=0.25cm,y=0.25cm,line width=.2ex,-stealth,#1] \draw (0,0) arc (150:-150:1 and 1);%
}
\newcommand{\EVert}[2]{%
\left[\begin{smallmatrix} #1 \\ #2 \end{smallmatrix}\right]
}
\newcommand{\TVert}[3]{%
\left[\begin{smallmatrix} #1 \\ #2 \\ #3 \end{smallmatrix}\right]
}
\usepackage{pgfplots}
\usepackage{tikz-network}
\pgfplotsset{compat=1.5}
\usetikzlibrary{intersections}

\usetikzlibrary{decorations.pathreplacing,decorations.markings}
\tikzset{
  on each segment/.style={
    decorate,
    decoration={
      show path construction,
      moveto code={},
      lineto code={
        \path [#1]
        (\tikzinputsegmentfirst) -- (\tikzinputsegmentlast);
      },
      curveto code={
        \path [#1] (\tikzinputsegmentfirst)
        .. controls
        (\tikzinputsegmentsupporta) and (\tikzinputsegmentsupportb)
        ..
        (\tikzinputsegmentlast);
      },
      closepath code={
        \path [#1]
        (\tikzinputsegmentfirst) -- (\tikzinputsegmentlast);
      },
    },
  },
  mid arrow/.style={postaction={decorate,decoration={
        markings,
        mark=at position .5 with {\arrow[#1]{stealth}}
      }}},
}

\usepackage{graphicx, epstopdf}
\usepackage[caption=false]{subfig}
\usepackage{multirow}
\definecolor{bananamania}{rgb}{0.98, 0.91, 0.71}
\definecolor{lavender}{rgb}{0.4470588235294118, 0.5294117647058824, 0.992156862745098}
\definecolor{burntsienna}{rgb}{0.91, 0.45, 0.32}
\definecolor{airforceblue}{rgb}{0.36, 0.54, 0.66}
\definecolor{liberty}{HTML}{5158BB}
\definecolor{junglegreen}{rgb}{0.16, 0.67, 0.53}
\Crefname{ALC@unique}{Line}{Lines}

\setlist[itemize]{leftmargin=*}
\setlist[enumerate]{leftmargin=*}

\begin{document}

\maketitle

\begin{abstract}
  Simplicial complexes are generalizations of classical graphs. Their homology groups are widely used to characterize the structure and the topology  of  data in e.g. chemistry, neuroscience, and transportation networks. 
  In this work we assume we are given a simplicial complex and that we can act on its underlying graph, formed by the set of 1-simplices, and we investigate the stability of its homology with respect to perturbations of the edges in such graph. Precisely, exploiting the isomorphism between the homology groups and the higher-order Laplacian operators, we propose a numerical method to compute the smallest graph perturbation sufficient to change the dimension  of the simplex's Hodge homology. Our approach is based on a matrix nearness problem formulated as a matrix differential equation, which requires an appropriate weighting and normalizing procedure for the boundary operators acting on the  Hodge algebra's homology groups. 
  We develop a bilevel optimization procedure suitable for the formulated matrix nearness problem and illustrate the method’s performance  on a variety of synthetic quasi-triangulation datasets and real-world transportation networks.

\end{abstract}



\section{Introduction}
\label{sec:introduction}
Models based on graphs are ubiquitous in the sciences and engineering as they allow us to model various complex systems in a unified form. Despite being widely and successfully used, graph-based models are limited to pairwise relationships and a variety of recent work has shown that  
many complex systems and datasets are better described by higher-order relations that go beyond pairwise interactions \cite{Battiston2020, Benson16,BickSchaub21}. 
Relational data is full of interactions that happen in groups. For example, friendship relations often involve groups that are larger than two individuals and triangles are important building blocks of relational data \cite{altenburger2018monophily,arrigo2020framework,nettasinghe2019diffusion}.  
Also in the presence of point-cloud data, directly modeling higher-order data interactions has led to improvements in numerous data mining settings, including  clustering~\cite{Benson16,fountoulakis2021local,tudisco2021nonlinear}, link prediction~\cite{arrigo2020framework,benson2018simplicial}, and ranking~\cite{benson2019three,tudisco2017node,tudisco2022core}.


A popular extension of dyadic models to the higher-order setting uses hypergraphs and relies on the transition from matrix-based to tensor-based models~\cite{Battiston2020,BickSchaub21}. Based on recent generalizations of graph spectral theory to tensor eigenproblems \cite{tensor_analysis_book,gautier_sigest_2023}, a number of graph methods based on matrices for e.g.\ clustering, community detection, centrality, opinion dynamics, have been extended to higher-order models based on tensors~\cite{arrigo2020framework,tudisco2021nonlinear,benson2019three,neuhauser2022consensus,neuhauser2023learning}.

Alongside hypergraphs and tensors, \emph{simplicial complexes} and \emph{higher-order Laplacians} are another standard model for higher-order interactions, 
where simplices of different order can connect an arbitrary number of nodes \cite{Battiston2020,BickSchaub21}. 
Higher-order Laplacians are linear operators that generalize the better-known graph Laplacian and provide key algebraic tools that allow to describe a simplicial complex and its structural properties. In particular, their kernels define a homology of the data and reveal fundamental topological properties such as connected components, holes, and voids \cite{lim2015hodge,chen2021decomposition}. 

In this work we are interested to quantifying the stability of such topological properties, with respect to edge perturbations. More precisely, given an initial simplicial complex $\mathcal K$ with a corresponding homology and an underlying graph $\mathcal G_{\mathcal K}$ formed by the 1-simplices of $\mc K$, we want to quantify how far is $\mathcal K$ from another simplicial complex with a homology group of strictly different dimension. Here we assume we can act on the edges of $\mathcal G_{\mathcal K}$ and ``being far''  is measured in terms of the least number (or, in the weighted case, the least weight) of edges of $\mc G_{\mc K}$ that can be eliminated to change the dimension of the chosen homology group. While this form of stability is reminiscent of the persistence of the homology, which is widely studied in the literature, we remark that the two are significantly different. Rather than starting with a dataset of points on which to build a chain of simplices and an associated persistency diagram, in our setting, we assume we are given an initial simplex as a result of a data assimilation and modeling process we have no access to. Thus, acting on the elementary pairwise information we are given, the 1-simplices, we want to quantify the robustness of the homology that characterizes the simplex at hand. While a great effort has been devoted in recent years to measure the presence and persistence  of simplicial homology \cite{ebli2019notion,otter2017roadmap}, much less is available about the stability of the homology classes with respect to data perturbation in the non-regular (non-geometric) setting \cite{chazal2014persistence}. 

The resulting problem can be formulated as a structured matrix nearness problem, which we approach by means of a  spectral objective function and the integration of the corresponding matrix-valued gradient flow. 
In order to make a sound mathematical formulation of the problem we aim to solve, and of the numerical model we design for its solution, we structure the remainder of the paper as follows: in \Cref{sec:simplicial-complexes-and-higher-order-laplacians} we review in detail the notion of simplicial complexes and the corresponding higher-order Laplacians. In \Cref{sec:weighted-case} and \Cref{subsec:connetedness} we  discuss how these operators may be extended to account for weighted higher-order node relations and formulate the corresponding stability problem in  \Cref{sec:nearest_complex} and \Cref{subsec:functional}. After introducing a suitable spectral functional whose minimization to zero mathematically translates the stability problem, in \Cref{sec:2lev} and \Cref{sec:algorithm} we present a two-level methodology whose inner iteration consists of a constrained matrix gradient flow on which is based our numerical method, and whose outer level tunes a suitable scalar parameter in order to solve a related scalar equation. Finally, we devote \Cref{sec:experiments} to illustrate the performance of the proposed numerical scheme on several example datasets.

\section{Simplicial complexes and higher-order relations} \label{sec:simplicial-complexes-and-higher-order-laplacians}

A graph $\mc G$ is a pair of sets $(\mc V, \mc E)$, where $\mc V=\{1,\dots,n\}$ is the set of vertices  and $\mc E \subset \mc  V \times \mc V$ is a set of unordered pairs representing the undirected edges of $\mc G$. We let $m$ denote the number of edges $\mc E = \{e_1,\dots,e_m\}$ and we assume them ordered lexicographically, with the convention that $i<j$ for any $\{i,j\}\in \mc E$. For brevity, we often write $ij$ in place of $\{i,j\}$ to denote the edge joining $i$ and $j$. Moreover, we assume no self-loops, i.e., $ii \notin \mc E$ for all $i \in \mc V$. 

A graph  only considers pairwise relations between the vertices. A simplicial complex $\mc K$ is a generalization of a graph that allows us to model connections involving an arbitrary number of nodes by means of higher-order simplices.
Formally,  a $k$-th order simplex (or $k$-simplex, briefly) is a set of $k+1$ vertices $\{i_0,i_1,\dots, i_k\}$  with the property that every subset of $k$ nodes itself is a $(k-1)$-simplex. Any $(k-1)$-simplex of a $k$-simplex is called a \textit{face}. The collection of all such simplices forms a simplicial complex $\mc K$, which therefore  essentially consists of a collection of sets of vertices such that every subset of the set in the collection is in the collection itself. Thus, a graph $\mc G$ can be thought of as the collection of $0$- and $1$- simplices:  the $0$-simplices form the nodes set of $\mc G$, while $1$-simplices form its edges. To emphasize this analogy, in the sequel we often specify that $\mc K$ is a simplicial complex on the vertex set $\mc V$.

Just like the edges of a graph, to any simplicial complex, we can associate an orientation (or ordering). To underline that an ordering has been fixed, we denote an ordered  $k$-simplex $\sigma$ using square brackets $\sigma = [i_0\dots i_k]$.  
In particular, as for the case of edges, 
we always assume the lexicographical ordering, unless specified otherwise.  That is, we assume that:
\begin{enumerate}[leftmargin=*]
    \item any $k$-simplex $[i_0 \dots i_k]$ in $\mc K$ is such that $i_0<\dots<i_k$;
    \item the $k$-simplices $\sigma_1,\sigma_2,\dots$  of $\mc K$ are ordered so that $\sigma_i\prec \sigma_{i+1}$ for all $i$, where $[i_0\dots i_k]\prec [i_0'\dots i_k']$ if and only if there exists $h$ such that $0\leq h\leq k$, 
    $i_j=i_j'$ for $j=0,\dots,h$ and $i_h<i_h'$.
\end{enumerate}
As for the edges, we often write $i_0\dots i_k$ in place of $[i_0\dots i_k]$ in this case.

\subsection{Homology, boundary operators and higher-order Laplacians}
Topological properties of a simplicial complex can be studied by considering boundary operators, higher-order Laplacians, and the associated homology.
Here we recall these concepts trying to emphasize their matrix-theoretic flavor. To this end, we first fix some further notation and recall the notion of a real $k$-chain.

\begin{definition}
  Assume $\mc K$ is a simplicial complex on the vertex set $\mc V$. For $k\geq 0$,  we denote the set of all the oriented $k$-th order simplices in $\mc K$ as $\mc V_k(\mc K)$ or simply $\mc V_k$. Thus, $\mc V_0=\mc V$ and $\mc V_1=\mc E$ form the underlying graph of $\mc K$, which we denote by $\mc G_{\mc K}=(\mc V,\mc E)=(\mc V_0,\mc V_1)$. 
\end{definition}

\begin{definition}  \label{def:chains}
  The formal real vector space spanned by  all the elements of $\mc V_k$ with real coefficients is  denoted by $C_k(\mc K)$. Any element of $C_k(\mc K)$, the formal linear combinations of simplices in  $\mc V_k$, is called a  \emph{$k$-chain}. 
\end{definition}
We remark that, in the graph-theoretic terminology, $C_0(\mc K)$ is   usually called the space of \emph{vertices' states}, while $C_1(\mc K)$ is usually called the space of \emph{flows} in the graph.

The chain spaces are finite vector spaces generated by the set of $k$-simplices. The boundary and co-boundary operators are particular linear mappings between $C_k$ and $C_{k-1}$, which in a way are the discrete analogous of high-order differential operators (and their adjoints) on continuous manifolds.  The boundary operator $\partial_k$ maps a $k$-simplex to an alternating sum of its $(k-1)$-dimensional faces obtained
by omitting one vertex at a time. Its  precise definition is recalled below, while \Cref{fig:bound_mat} provides an illustrative example of its action. 

\begin{definition}
  \label{def:bound_oper}
  Let $k\geq 1$. Given a simplicial complex $\mc K$ over the set $\mc V$, the  \emph{boundary operator} $\partial_k \colon C_k(\mc K) \mapsto C_{k-1}(\mc K)$ maps every ordered $k$-simplex  $[i_0 \ldots i_k]\in C_k(\mc K)$  to the following alternated sum of its faces:
  \[
    \partial_k [ i_0 \ldots i_k]=\sum_{j=0}^k (-1)^{j} [i_0 \ldots i_{j-1} i_{j+1} \ldots i_k] \in C_{k-1}(\mc K)
  \]
\end{definition}

As we assume the $k$-simplices in $\mc V_k$ are ordered lexicographically, we can fix a canonical basis for $C_k(\mc K)$ and we can represent  $\partial_k$ as a matrix $B_k$ with respect to such basis. In fact, once the ordering is fixed, $C_k(\mc K)$ is isomorphic to $\mathbb R^{\mc V_k}$, the space of functions from $\mc V_k$ to $\mathbb R$ or, equivalently,  the space of real vectors with $|\mc V_k|$ entries. Thus, $B_k$ is a $|\mc V_{k-1}|\times |\mc V_k|$ matrix and $\partial_k^*$ coincides with $B_k^\top$. 
We shall always assume the canonical basis for $C_k(\mc K)$ is fixed in this way and we will deal exclusively with the matrix representation $B_k$ from now on. 
An example of $B_k$ for $k=1$ and $k=2$ is shown in \Cref{fig:bound_mat}. 
\begin{figure}[t]
  \centering
 \scalebox{1.0}{ 
 \begin{tikzpicture}
  \fill [opacity=0.1,red]    (0, 0.5) -- (4/3, 0.5) --  (2/3, 11/6) -- cycle;
  \fill [opacity=0.1,red]    (7/3, 2.5) -- (10/3, 2.5) --  (3, 3.5) -- cycle;
  \fill [opacity=0.2,red]    (7/3, 2.5) -- (8/3, 7/6) --  (10/3, 2.5) -- cycle;
 \Vertex[x=0, y=0.5, label=1, size=0.3]{v1}
 \Vertex[x=2/3, y=11/6, label=2, size=0.3]{v2}
 \Vertex[x=4/3, y=0.5, label=3, size=0.3]{v3}
 \Vertex[x=7/3, y=2.5, label=4, size=0.3]{v4}
 \Vertex[x=8/3, y=7/6, label=5, size=0.3]{v5}
 \Vertex[x=10/3, y=2.5, label=6, size=0.3]{v6}
 \Vertex[x=3, y=3.5, label=7, size=0.3]{v7}
 \Edge[Direct](v1)(v2)
 \Edge[Direct](v1)(v3)
 \Edge[Direct](v2)(v3)
 \Edge[Direct](v2)(v4)
 \Edge[Direct](v3)(v5)
 \Edge[Direct](v4)(v5)
 \Edge[Direct](v4)(v6)
 \Edge[Direct](v4)(v7)
 \Edge[Direct](v5)(v6)
 \Edge[Direct](v6)(v7)

  \node at (2/3,1){\AxisRotatorMirror[rotate=90]};
  \node at (8.25/3,12.5/6){\AxisRotator[rotate=-90]};
  \node at (8.75/3,17/6){\AxisRotator[rotate=-30]};
  
  \Edge[Direct, color=burntsienna, bend=-60, style={dashed}](1/8, 1)(-0.25, -0.25)
  
  \node at (7.5, 3){
   $
   B_1 =  \tiny\arraycolsep=1.4pt 
   \left(
      \begin{array}{c|cccccccccc}
           & \EVert{1}{2} & \EVert{1}{3} & \EVert{2}{3} & \EVert{2}{4} & \EVert{3}{5} &\EVert{4}{5} & \EVert{4}{6} & \EVert{4}{7} & \EVert{5}{6} & \EVert{6}{7} \\
       \hline
       [1] & -1 & -1 & 0 & 0 & 0 & 0 & 0 & 0 & 0 & 0 \cr 
       [2] & 1 & 0 & -1 & -1 & 0 & 0 & 0 & 0 & 0 & 0 \cr
       [3] & 0 & 1 & 1 & 0 & -1 & 0 & 0 & 0 & 0 & 0 \cr
       [4] & 0 & 0 & 0 & 1 & 0 & -1 & -1 & -1 & 0 & 0 \cr
       [5] & 0 & 0 & 0 & 0 & 1 & 1 & 0 & 0 & -1 & 0 \cr
       [6] & 0 & 0 & 0 & 0 & 0 & 0 & 1 & 0 & 1 & -1 \cr
       [7] & 0 & 0 & 0 & 0 & 0 & 0 & 0 & 1 & 0 & 1 
      \end{array}
    \right)$
  };

  \node at (7, 0.5){
  $
  B_2 = \tiny\arraycolsep=1.4pt 
  \left(
      \begin{array}{c|ccc}
        & \TVert{1}{2}{3} & \TVert{4}{5}{6} & \TVert{4}{6}{7} \\
        \hline
        [1,2] & 1 & 0 & 0 \cr
        [1,3] & -1 & 0 & 0 \cr
        [2,3] & 1 & 0 & 0 \cr
        [2,4] & 0 & 0 & 0 \cr
        [3,5] & 0 & 0 & 0 \cr
        [4,5] & 0 & 1 & 0 \cr
        [4,6] & 0 & -1 & 1 \cr
        [4,7] & 0 & 0 & -1 \cr
        [5,6] & 0 & 1 & 0 \cr
        [6,7] & 0 & 0 & 1 \cr
      \end{array}
  \right)
  $
  };
  \node at (2, -0.5){ \small 
  $ \partial_2 ( [ 1, 2, 3] ) = [1, 2] - [1, 3] + [2, 3]$
  };
\end{tikzpicture} 
 }
 \caption{Left-hand side panel: example of simplicial complex $\mc K$ on $7$ nodes, and of the action of $\partial_2$ on the 2-simplex $[1,2,3]$; $2$-simplices included in the complex are shown in red, arrows correspond to the orientation. Panels on the right: matrix forms $B_1$ and $B_2$ of boundary operators $\partial_1$ and $\partial_2$ respectively.  \label{fig:bound_mat}}
\end{figure}
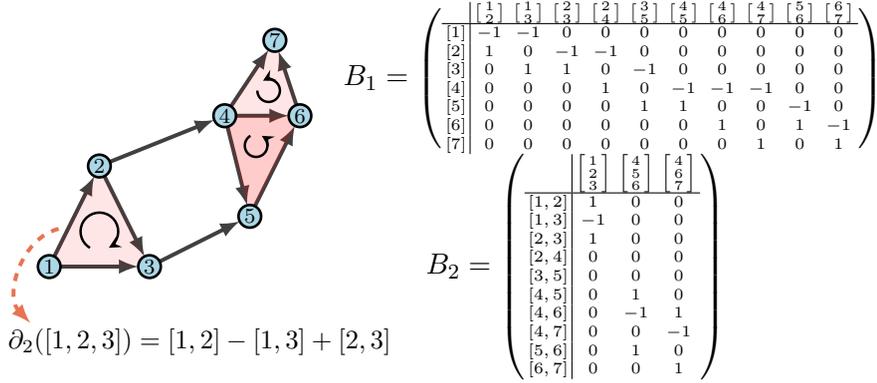

A direct computation shows that  the following fundamental   identity holds (see e.g.~\cite[Thm.~5.7]{lim2015hodge})
\begin{equation}
    B_k  B_{k+1} = 0\label{eq:Bk_Bk+1_0}
\end{equation}
for any $k$. This identity  is  also known in the continuous case as the Fundamental Lemma of Homology and it allows us to define a homology group associated with each $k$-chain. In fact, \Cref{eq:Bk_Bk+1_0} implies in particular that $\mathrm{im} \, B_{k+1} \subset \ker B_k$, so that the $k$-th homology group is correctly defined: 
\begin{equation*}
  \mc H_k=\faktor{\ker B_k}{\mathrm{im} \, B_{k+1}}
\end{equation*}

The dimensionality of the $k$-th homology group is known as \emph{$k$-th Betti's number} $\beta_k=\dim \mc H_k$, while the elements of $\mc H_k$ correspond to so-called $k$-dimensional holes in the simplicial complex. For example, $\mc H_0$, $\mc H_1$  and $\mc H_2$ describe \emph{connected components}, \emph{holes} and \emph{three-dimensional voids} respectively.
By standard algebraic passages one sees that $\mc H_k$ is isomorphic  to $\ker \left( B^\top_k B_k + B_{k+1} B^\top_{k+1}  \right)$. Thus, the number of $k$-dimensional holes corresponds to the dimension of the kernel of a linear operator, which is known as \textit{$k$-th order Laplacian} or \textit{higher-order Laplacian} of the simplicial complex $\mc K$.

\begin{definition}
  \label{def:lap}
Given a simplicial complex $\mc K$ and the boundary operators $B_k$ and $B_{k+1}$,
the \emph{$k$-th order Laplacian} $L_k$ of $\mc K$ is the $|\mc V_k|\times |\mc V_k|$ matrix defined as:
  \begin{equation}
    \label{eq:def_lap}
    L_k=B_k^\top B_k + B_{k+1} B_{k+1}^\top
  \end{equation}
  \end{definition}
  In particular, we remark that:
  \begin{itemize}[leftmargin=*]
    \item the $0$-order Laplacian is the standard \textit{combinatorial graph Laplacian}  $\La=B_1 B_1^\top \in \mathbb R^{n\times n}$, whose diagonal entries consist of the degrees of the corresponding vertices (i.e.\ the number of $1$-simplices each vertex belongs to), while the off-diagonal $(L_0)_{ij}$ is equal to $-1$ if either $ij$ is a $1$-simplex, and it is zero otherwise; 
    \item the $1$-order Laplacian is known as  \textit{Hodge Laplacian} $\Lap =  B_1^\top B_1 + B_2  B_2^\top   \in \mathbb R^{m\times m}$.     Similarly to the $0$-order case $\La$, one can describe the entries of $\Lap$ in terms of the structure of the simplicial complex, see e.g.\ \cite{muhammad2006control}.
  \end{itemize}

  \subsubsection{Connected components and holes}

The boundary operators $B_k$ on $\mc K$ are directly connected with discrete notions of differential operators on the graph. In particular, $B_1$, $B_1^\top$, and $B_2^\top$ are the graph's divergence, gradient, and curl operators, respectively. We refer to \cite{lim2015hodge} for more details. 
As  $\mc H_k$ is isomorphic to $\ker L_k$, we have that the following \textit{Hodge decomposition} of $\mathbb R^{\mc V_k}$ holds
\begin{equation}\label{eq:hodge-decomposition}
    \mathbb R^{\mc V_k} = \im B_{k+1} \oplus \im B_k^\top \oplus \ker L_k = \im \bar B_{k+1} \oplus \im \bar  B_k^\top \oplus \ker \bar L_k \, .
\end{equation} 
Thus, the space of vertex states $\mathbb R^{\mc V_0}$ can be decomposed as $\mathbb R^{\mc V_0}=\im B_1\oplus\ker L_0$, the sum of divergence-free vectors  and harmonic vectors, which correspond to the connected components of the graph. In particular,  for a connected graph, $\ker L_0$ is one-dimensional and consists of entry-wise constant vectors. Thus, for a connected graph, $\im B_1$ is the set of vectors whose entries sum up to zero. 

Similarly, the space of flows on graph's edges $\mathbb R^{\mc V_1}$ can be decomposed as $\mathbb R^{\mc V_1}=\im B_{2} \oplus \im B_1^\top \oplus \ker L_1$. Thus, each flow can be decomposed into its gradient part $\im B_1^\top$, which consists of flows with zero cycle sum, its curl part $\im B_2$, which consists of circulations around order-2 simplices in $\mc K$, and its harmonic part $\ker L_1$, which represents 1-dimensional holes defined as  global circulations modulo the curl flows. 

While $0$-dimensional holes are easily understood as the connected components of the graph $\mc G_{\mc K}$, a notion of ``holes in the graph'' $\mc G_{\mc K}$ corresponds to $1$-dimensional holes in $\mc K$. This terminology comes from the analogy with the continuous case. In fact, if the graph is obtained as a discretization of a continuous manifold, harmonic functions in the homology group  $\mc H_1$ would correspond to the  holes in the manifold, as illustrated in \Cref{fig:example_holes}.  Moreover, the Hodge Laplacian of a simplicial complex built on $N$ randomly sampled points in the manifold converges in the thermodynamic limit to its continuous counterpart, as $N\to \infty$,  \cite{chen2021}. 

\begin{figure}[t]
  \centering 
  {
  \scalebox{1.0}{
    \begin{tikzpicture}
      \draw [black, line width=1, fill=black!20!white] plot [smooth cycle] coordinates {(0,0) (1,1) (3,1) (3,0) (2,-1)};
      \draw [black, line width=1, fill=white] plot [smooth cycle] coordinates {(2.5, 0.8)  (1.5, 0.55) (2, 0.3)     };
  \draw [black!50!white, dashed, line width=1, fill=black!20!white] plot [smooth cycle] coordinates {(5,0) (6,1) (8,1) (8,0) (7,-1)};
      \draw [black!50!white, dashed, line width=1, fill=white] plot [smooth cycle] coordinates {(7, 0.3) (6.5, 0.55) (7.5, 0.8)};
  \draw [black, line width=0.5, fill=black!20!white, circle, dashed  ] (10,0) -- (10.6, 0) -- (10.3, 0.4) -- cycle;
  \draw [black, line width=0.5, fill=black!20!white, circle, dashed  ] (10,0) -- (10.65, 0) -- (10.7, -0.32) -- cycle;
  \draw [black, line width=0.5, fill=black!20!white, circle, dashed  ] (10.6,0) -- (10.3, 0.4) -- (11, 1) -- cycle;
  \draw [black, line width=0.5, fill=black!20!white, circle, dashed  ] (10.6,0) -- (11.1, 0.4) -- (11, 1) -- cycle;
  \draw [black, line width=0.5, fill=black!20!white, circle, dashed  ] (11.5,0.55) -- (11.1, 0.4) -- (11, 1) -- cycle;
  \draw [black, line width=0.5, fill=black!20!white, circle, dashed  ] (11.5,0.55) -- (12, 1) -- (11, 1) -- cycle;
  \draw [black, line width=0.5, fill=black!20!white, circle, dashed  ] (11.5,0.55) -- (12.5, 0.8) -- (12, 1) -- cycle;
  \draw [black, line width=0.5, fill=black!20!white, circle, dashed  ] (13,1) -- (12.5, 0.8) -- (12, 1) -- cycle;
  \draw [black, line width=0.5, fill=black!20!white, circle, dashed  ] (13,1) -- (12.5, 0.8) -- (13, 0) -- cycle;
  \draw [black, line width=0.5, fill=black!20!white, circle, dashed  ] (12,0.3) -- (12.5, 0.8) -- (13, 0) -- cycle;
  \draw [black, line width=0.5, fill=black!20!white, circle, dashed  ] (12,-1) -- (12.1, -0.55) -- (13, 0) -- cycle;
  \draw [black, line width=0.5, fill=black!20!white, circle, dashed  ] (12,-1) -- (12.1, -0.55) -- (11.5, -0.2) -- cycle;
  \draw [black, line width=0.5, fill=black!20!white, circle, dashed  ] (10.7,-0.32) -- (12, -1) -- (11.5, -0.2) -- cycle;
  \draw [black, line width=0.5, fill=black!20!white, circle, dashed  ] (10.7,-0.32) -- (10.6, 0) -- (11.5, -0.2) -- cycle;
  \draw [black, line width=0.5, fill=black!20!white, circle, dashed  ] (11.1, 0.4) -- (10.6, 0) -- (11.5, -0.2) -- cycle;
  \draw [black, line width=0.5, fill=black!20!white, circle, dashed  ] (11.1, 0.4) -- (11.5, 0.55) -- (11.5, -0.2) -- cycle;
  \draw [black, line width=0.5, fill=black!20!white, circle, dashed  ] (12, 0.3) -- (11.5, 0.55) -- (11.5, -0.2) -- cycle;
  \draw [black, line width=0.5, fill=black!20!white, circle, dashed  ] (12, 0.3) -- (12.1, -0.55) -- (11.5, -0.2) -- cycle;
  \draw [black, line width=0.5, fill=black!20!white, circle, dashed  ] (12, 0.3) -- (12.1, -0.55) -- (13, 0) -- cycle;
  \node [draw=black,fill=black, shape=circle, scale=0.3] at (10, 0) {};
  \node [draw=black,fill=black, shape=circle, scale=0.3] at (11.1, 0.4) {};
  \node [draw=black,fill=black, shape=circle, scale=0.3] at (11.5, 0.55) {};
  \node [draw=black,fill=black, shape=circle, scale=0.3] at (11.5, -0.2) {};
  \node [draw=black,fill=black, shape=circle, scale=0.3] at (12, 0.3) {};
  \node [draw=black,fill=black, shape=circle, scale=0.3] at (12.1, -0.55) {};
  \node [draw=black,fill=black, shape=circle, scale=0.3] at (13, 0) {};
  \node [draw=black,fill=black, shape=circle, scale=0.3] at (10.7, -0.32) {};
  \node [draw=black,fill=black, shape=circle, scale=0.3] at (10.6, 0) {};
  \node [draw=black,fill=black, shape=circle, scale=0.3] at (12, -1) {};
  \node [draw=black,fill=black, shape=circle, scale=0.3] at (12.5, 0.8) {};
  \node [draw=black,fill=black, shape=circle, scale=0.3] at (10.3, 0.4) {};
  \node [draw=black,fill=black, shape=circle, scale=0.3] at (11, 1) {};
  \node [draw=black,fill=black, shape=circle, scale=0.3] at (12, 1) {};
  \node [draw=black,fill=black, shape=circle, scale=0.3] at (13, 1) {};
  \node [draw=black,fill=black, shape=circle, scale=0.3] at (5, 0) {};
  \node [draw=black,fill=black, shape=circle, scale=0.3] at (6.1, 0.4) {};
  \node [draw=black,fill=black, shape=circle, scale=0.3] at (6.5, 0.55) {};
  \node [draw=black,fill=black, shape=circle, scale=0.3] at (6.5, -0.2) {};
  \node [draw=black,fill=black, shape=circle, scale=0.3] at (7, 0.3) {};
  \node [draw=black,fill=black, shape=circle, scale=0.3] at (7.1, -0.55) {};
  \node [draw=black,fill=black, shape=circle, scale=0.3] at (8, 0) {};
  \node [draw=black,fill=black, shape=circle, scale=0.3] at (5.7, -0.32) {};
  \node [draw=black,fill=black, shape=circle, scale=0.3] at (5.6, 0) {};
  \node [draw=black,fill=black, shape=circle, scale=0.3] at (7, -1) {};
  \node [draw=black,fill=black, shape=circle, scale=0.3] at (7.5, 0.8) {};
  \node [draw=black,fill=black, shape=circle, scale=0.3] at (5.3, 0.4) {};
  \node [draw=black,fill=black, shape=circle, scale=0.3] at (6, 1) {};
  \node [draw=black,fill=black, shape=circle, scale=0.3] at (7, 1) {};
  \node [draw=black,fill=black, shape=circle, scale=0.3] at (8, 1) {};

    \end{tikzpicture}
    }
  }
  \caption{Continuous and analogous discrete manifolds with one $1$-dimensional hole ($\dim \bar{\mc H}_1=1$). Left pane: the continuous manifold; center pane: the discretization with mesh vertices; right pane: a simplicial complex built upon the mesh. Triangles in the simplicial complex $\mc K$ are colored gray (right). 
  \label{fig:example_holes}
  }
\end{figure}
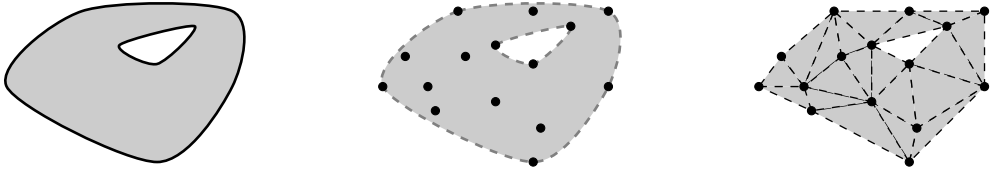

\subsection{Normalized and weighted higher-order Laplacians}\label{sec:weighted-case}

In the classical graph setting, a normalized and weighted version of the Laplacian matrix is very often employed in applications.  
From a matrix theoretic point of view, having a weighted graph corresponds to allowing arbitrary nonnegative entries in the adjacency matrix defining the graph. In terms of boundary operators,  this coincides with a positive diagonal rescaling. Analogously, the normalized Laplacian is defined by applying a diagonal congruence transformation to the standard Laplacian using the node weights. We briefly review these two constructions below.

    Let $\mc G=(\mc V,\mc E)$ be a graph with positive node and edge weight functions $w_0:\mc V \to \mathbb R_{++}$ and  $w_1:\mc E\to \mathbb R_{++}$, respectively. Define the $|\mc V|\times |\mc V|$ and $|\mc E|\times |\mc E|$ diagonal matrices
    $
    W_0 = \mathrm{Diag}\{w_0(v_i)^{1/2}\}_{i=1}^n$ and $W_1 = \mathrm{Diag}\{w_1(e_i)^{1/2}\}_{i=1}^m
    $. 
    Then $\bar B_1 = W_0^{-1}B_1W_1$ is the normalized and weighted boundary operator of $\mc G$ and we have that 
    \begin{equation}\label{eq:normalized_L0}
          \bar L_0 = \bar B_1\bar B_1^\top = W_0^{-1}B_1 W_1^2 B_1^\top W_0^{-1}
    \end{equation}
    is the normalized weighted graph Laplacian of $\mc G$. 

In particular, note that, as for $L_0$, the entries of $\bar L_0$ uniquely characterize the graph topology, in fact we have
    $$
    (\bar L_0)_{ii} = \frac{\deg(i)}{w_0(i)}, \qquad (\bar L_0)_{ij} = \begin{cases}-\frac{w_1(ij)}{\sqrt{w_0(i)w_0(j) }} & ij \in \mc E \\ 0 & \text{otherwise} \end{cases}, \text{ for } i\neq j
    $$
    where $\deg(i)$ denotes  the (weighted) degree of  node $i$, i.e.\ 
    $\deg(i) = \sum_{e\in \mc E: i\in e}w_1(e)$. 

While the definition of $k$-th order Laplacian is well-established for the case of unweighted edges and simplices, a notion of weighted and normalized $k$-th order Laplacian is not universally available and it might depend on the application one has at hand.  For example, different definitions of weighted Hodge Laplacian are considered in \cite{chen2021decomposition,horak2013spectra,lim2015hodge,schaub2020random}.

At the same time, we notice that the notation used  in \Cref{eq:normalized_L0} directly  generalizes to higher orders. Thus, we propose the following notion of normalized and weighted $k$-th Laplacian 

\begin{definition}\label{def:normalized_L1}
  Let ${\mc K}$ be a simplicial complex and  let $w_{k}:\mc V_k\to \mathbb R_{++}$ be a positive-valued weight function on the $k$-simplices of $\mc K$. Define the diagonal matrix $W_k=\mathrm{Diag}\big\{w_k(\sigma_i)^{1/2}\big\}_{i=1}^{|\mc V_k|}$. 
  Then, $\bar B_k= W_{k-1}^{-1}B_kW_k$ is the normalized and weighted $k$-th boundary operator, to which corresponds the normalized and weighted $k$-th Laplacian 
  \begin{equation} \label{eq:bar_L_k}
       \begin{aligned}
  \bar L_k &= \bar B_k^\top  \bar B_k + \bar B_{k+1} \bar B_{k+1}^\top \\
     &= W_kB_k^\top  W_{k-1}^{-2} B_k  W_k + W_k^{-1}B_{k+1} W_{k+1}^2 B_{k+1}^\top W_k^{-1}.
 \end{aligned}
  \end{equation}
\end{definition}

Note that, from the definition $\bar B_k= W_{k-1}^{-1}B_kW_k$ and \Cref{eq:Bk_Bk+1_0}, we immediately have that  $\bar B_k\bar B_{k+1}=0$. Thus, the group $\bar {\mc H}_k = \ker \bar B_k / \im \bar B_{k+1}$ is well defined for any choice of positive weights $w_k$ and is isomorphic to $ \ker \bar L_k$. While the homology group may depend on the weights,  we observe below that its dimension does not. Precisely, we have 
\begin{proposition}
 \label{thm:wHomGroup}
     The dimension of the homology groups of $\mc K$ is not affected by the weights of its $k$-simplices. Precisely, if $W_k$ are positive diagonal matrices, we have
     \begin{equation}
         \dim \ker \bar B_k = \dim \ker B_k, \quad \dim \ker \bar B_k^\top = \dim \ker B_k^\top, \quad \dim \bar{\mc H}_k = \dim \mc H_k\, .
     \end{equation}
     Moreover, $\ker B_k = W_k \ker \bar B_k$ and $\ker B_k^\top = W_{k-1}^{-1} \ker \bar B_k^\top$.
\end{proposition}

\begin{proof}
Since $W_k$ is an invertible diagonal matrix, 
$$
    \bar B_k \vec x = 0 \iff W_{k-1}^{-1}B_k W_k \vec x =0 \iff B_k W_k \vec x =0. 
$$
Hence, if $\vec x \in \ker \bar B_k$, then $W_k \vec x \in \ker B_k$, and, since $W_k$ is bijective, $\dim \ker \bar B_k = \dim \ker B_k$. Similarly, one observes that $\dim \ker \bar B_k^\top = \dim \ker B_k^\top$.
Moreover, since $\bar B_k \bar B_{k+1} =0$, then $\im \bar B_{k+1} \subseteq \ker \bar B_k$ and $\im \bar B_k^\top \subseteq \ker \bar B_{k+1}^\top$. This yields $\ker \bar B_k \cup \ker \bar B_{k+1}^\top = \mathbb{R}^{\mc V_k} = \ker B_k \cup \ker  B_{k+1}^\top$. Thus, for the homology group it holds $\dim \mc \bar{H}_k = \dim \mc H_k$.
\end{proof}

\subsection{Principal spectral inheritance}

Before moving on to the next section, we recall here a relatively direct  but important spectral property that connects the spectra of the $k$-th and $(k+1)$-th order Laplacians.

\begin{theorem}[
Spectral inheritance of higher-order Laplacians
]\label{thm:inherit}
  Let $\bar L_k^{down}=\bar B_k^\top \bar B_k$ and $\bar L_k^{up}=\bar B_{k+1} \bar B_{k+1}^\top$. Then:
  \begin{enumerate}
    \item $\sigma_+(\bar L_k^{up})=\sigma_+(\bar L_{k+1}^{down})$, where $\sigma_+(\cdot)$ denotes the set of positive eigenvalues;
    \item for any $ \mu \in \sigma_+(\bar L_k) $,  either $ \mu \in \sigma_+ (\bar L_k^{up}) $ or the corresponding eigenvector $ \vec v  \in \ker \bar L_k^{up}$. Similarly, for any $ \nu \in \sigma_+(\bar L_{k+1}) $, either $ \nu \in \sigma_+ (\bar L_{k+1}^{down}) $ or the corresponding eigenvector $ \vec u  \in \ker \bar L_{k+1}^{down}$, and
    \[
    \bar B_k^\top \bar B_k \vec v  = \mu \vec v, \qquad \bar B_{k+2} \bar B_{k+2}^\top \vec u = \nu \vec u\, . 
    \]
  \end{enumerate}
\end{theorem}
\begin{proof}
  For $ \mu > 0 $ it is sufficient to note that if $ ( \mu, \vec x )$ is an eigenpair for $ \bar L_k^{up} $, then $ ( \mu, \mu^{-1/2} \bar B_{k+1}^\top \vec x ) $ is an eigenpair for $ \bar L_{k+1}^{down} $. Similarly, if $ (\mu, \vec y)$ is an eigenpair of $ \bar L_{k+1}^{down}$, thens $ ( \mu, \mu^{-1/2} \bar B_{k+1} \vec y ) $ is the corresponding eigenpair of $ \bar L_k^{up}$, yielding $(1)$. The statement in (2) follows immediately from the Hodge decomposition (\ref{eq:hodge-decomposition}). 
\end{proof}

In other words, the variation of the spectrum of the $k$-th Laplacian when moving from one order to the next one works as follows: 
the down-term $\bar L_{k+1}^{down}$ inherits the positive part of the spectrum from the up-term of  $\bar L_k^{up}$; the  eigenvectors corresponding to the inherited positive part of the spectrum lie in the kernel of $\bar L_{k+1}^{up}$; at the same time, the ``new'' up-term $\bar L_{k+1}^{up}$ has a new, non-inherited, part of the positive spectrum (which, in turn, lies in the kernel of the $(k+2)$-th down-term).

In particular, we notice that for $k = 0$, since $B_0=0$ and $\bar L_0=\bar L_0^{up}$, the  theorem yields $\sigma_+ (\bar L_0 ) = \sigma_+ (\bar{L_1}^{down}) \subseteq \sigma_+(\bar L_1)$. In other terms, the positive spectrum of the $\bar L_0$ is inherited by the spectrum of $\bar L_1$ and the remaining (non-inherited) part of $\sigma_+(\bar L_1)$ coincides with $\sigma_+(\bar L_1^{up})$. 
\Cref{fig:thm_spct_ill} provides an  illustration of the statement of  \Cref{thm:inherit} for $k = 0$.
  \begin{figure}[t]
    \centering
    { 
    \begin{tikzpicture}
      \node[draw] at (0,0) {0};
      \node[draw] at (0.5,0) {0};
      \node at (1, 0) {$\cdots$};
      \node[draw] at (1.5,0) {0};
      \node[draw, fill=bananamania] at (2.15,0) {\tiny{$\lambda_1$}};
      \node[draw, fill=bananamania] at (2.65,0) {\tiny{$\lambda_2$}};
      \node[draw, fill=burntsienna] at (3.15,0) {\tiny{$\lambda_3$}};
      \node[draw, fill=bananamania] at (3.65,0) {\tiny{$\lambda_4$}};
      \node[draw, fill=burntsienna] at (4.15,0) {\tiny{$\lambda_5$}};
      \node[draw, fill=burntsienna] at (4.65,0) {\tiny{$\lambda_6$}};
      \node[draw, fill=bananamania] at (5.15,0) {\tiny{$\lambda_7$}};
      \node[draw, fill=bananamania] at (5.65,0) {\tiny{$\lambda_8$}};
      \node[draw, fill=burntsienna] at (6.15,0) {\tiny{$\lambda_9$}};
      \node[draw, fill=bananamania] at (6.65,0) {\tiny{$\lambda_{10}$}};
      \node at (8.5, 0) {$\leftarrow \quad \sigma (\bar L_1)$\phantom{$\bar B_1$}};
    
      \node[draw] at (0,-0.7) {0};
      \node[draw] at (0.5,-0.7) {0};
      \node at (1, -0.7) {$\cdots$};
      \node[draw] at (1.5,-0.7) {0};
      \node[draw, fill=bananamania] at (2.15,-0.7) {\tiny{$\lambda_1$}};
      \node[draw, fill=bananamania] at (2.65,-0.7) {\tiny{$\lambda_2$}};
      \node[draw] at (3.15,-0.7) {0};
      \node[draw, fill=bananamania] at (3.65,-0.7) {\tiny{$\lambda_4$}};
      \node[draw] at (4.15,-0.7) {0};
      \node[draw] at (4.65,-0.7) {0};
      \node[draw, fill=bananamania] at (5.15,-0.7) {\tiny{$\lambda_7$}};
      \node[draw, fill=bananamania] at (5.65,-0.7) {\tiny{$\lambda_8$}};
      \node[draw] at (6.15,-0.7) {0};
      \node[draw, fill=bananamania] at (6.65,-0.7) {\tiny{$\lambda_{10}$}};
      \node at (8.5, -0.7) {$\leftarrow \quad \sigma (\bar B_1^T \bar B_1)$};
    
      \node[draw] at (0,-1.4) {0};
      \node[draw] at (0.5,-1.4) {0};
      \node at (1, -1.4) {$\cdots$};
      \node[draw] at (1.5,-1.4) {0};
      \node[draw] at (2.15,-1.4) {0};
      \node[draw] at (2.65,-1.4) {0};
      \node[draw, fill=burntsienna] at (3.15,-1.4) {\tiny{$\lambda_3$}};
      \node[draw] at (3.65,-1.4) {0};
      \node[draw, fill=burntsienna] at (4.15,-1.4) {\tiny{$\lambda_5$}};
      \node[draw, fill=burntsienna] at (4.65,-1.4) {\tiny{$\lambda_6$}};
      \node[draw] at (5.15,-1.4) {0};
      \node[draw] at (5.65,-1.4) {0};
      \node[draw, fill=burntsienna] at (6.15,-1.4) {\tiny{$\lambda_9$}};
      \node[draw] at (6.65,-1.4) {0};
      \node at (8.5, -1.4) {$\leftarrow \quad \sigma (\bar B_2 \bar B_2^T)$};
    
      \draw [
        thick,
        decoration={
            brace,
            mirror,
            raise=0.25cm
        },
        decorate
      ] (-0.25, -1.4) -- (1.75, -1.4) 
    node [pos=0.5,anchor=north,yshift=-0.25cm] {holes}; 
      \draw[pattern=north west lines] (1.75,0.4) rectangle (1.85, -1.8);
      \node[draw, align=center, fill=bananamania] at (1.8,-2.1) {$\mu$};
    \end{tikzpicture}
    }
    \caption{Illustration for the principal spectrum inheritance (\Cref{thm:inherit}) in case $k=0$: spectra of $\bar L_1$, $\bar L_1^{down}$ and $\bar L_1^{up}$ are shown. Colors signify the splitting of the spectrum, $\lambda_i>0 \in \sigma(\bar L_1)$ ; all yellow eigenvalues are inherited from $\sigma_+(\bar L_0)$; red eigenvalues belong to the non-inherited part. Dashed barrier $\mu$ signifies the penalization threshold (see the target functional in \Cref{subsec:functional}) preventing homological pollution (see \Cref{subsec:connetedness}). }
    \label{fig:thm_spct_ill}
  \end{figure}
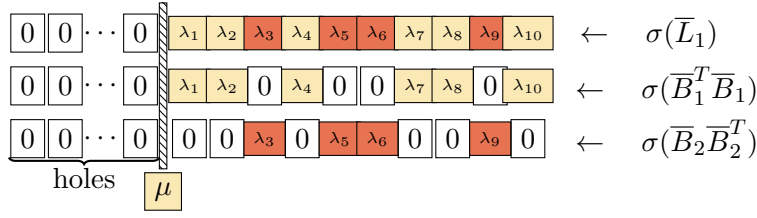

\section{Problem setting: Nearest complex with different homology}\label{sec:nearest_complex}
Suppose we are given a simplicial complex $\mc K$ on the vertex set $\mc V$, with simplex weight functions $w_0,w_1,\dots$, and let $\beta_k=\dim \mc H_k = \dim \bar{\mc H}_k$ the dimension of its $k$-homology. We aim at finding 
the closest simplex on the same vertex set $\mc V$, with a strictly larger number of holes. As it is the most frequent higher-order Laplacian appearing in applications and since this provides already a large number of numerical complications, we focus here only on the Hodge Laplacian case: given the simplicial complex $\mc K = (\mc V_0,\mc V_1,\mc V_2, \dots)$, we look for the smallest perturbation of the edges $\mc V_1$ which increases the dimension of $\mc H_1$. More precisely:

\begin{problem}\label{problem:1}
Let $\mc K$ be a simplicial complex of order at least $2$ with associated edge
weight function $w_1$ and corresponding diagonal weight matrix $W_1$, and let $\beta_1(W_1)$ be the dimension of the homology group corresponding to the weights in $W_1$. For  $\varepsilon>0$,  let 
\begin{equation*}
   \begin{aligned}
\Omega(\varepsilon) & =\Big\{\text{diagonal matrices $W$ such that $\|W\|= \varepsilon$}\Big\}, \\ 
\Pi(W_1) & =  \Big\{\text{diagonal matrices $W$ such that $W_1 + W \ge 0 $}\Big\}.
\end{aligned} 
\end{equation*}
In other words,  $\Omega(\eps)$ is an $\eps$-sphere and $\Pi(W_1)$ allows only non-negative simplex weights. We look for the smallest perturbation $\varepsilon$ such that there exists a weight modification $\delta W_1\in \Omega(\varepsilon) \cap \Pi(W_1)$ such that $\beta_1(W_1)<\beta_1(W_1+\delta W_1)$. 
\end{problem}
Here, and throughout the paper, $\|X\|$ denotes either the Frobenius norm if $X$ is a matrix, or the Euclidean norm if $X$ is a vector. Note that, as we are looking for the smallest $\eps$, the equality $\|W\|=\eps$ is an obvious choice, as opposed to $\|W\|\leq \eps$. 

As the dimension $\beta_1$ coincides with the dimension of the kernel of $\bar L_1$, we approach this problem through the minimization of a functional  based on the spectrum of the $0$-th and $1$-st Laplacian of the simplicial complex. In order to define such functional, we first make a number of considerations. 

Note that, due to \Cref{thm:wHomGroup}, the dimension of the first homology group does not change when the edge weights are perturbed, as long as all the weights remain positive. Thus, in order to find the desired perturbation $\delta W_1$, we need to set some of the initial weights to zero. This operation creates several potential issues we need to 
address, as discussed next.

First, setting an edge to zero implies that one is formally removing that edge from the simplicial complex. As the simplicial complex structure needs to be maintained, when doing so we need to set to zero also the weight of any $2$-simplex that contains that edge. For this reason, if $\tilde w_1(e) = w_1(e)+\delta w_1(e)$ is the new edge weight function, we require the weight function of the  $2$-simplices to change into $\tilde w_2$, defined as
\[
\tilde w_2(i_1i_2i_3) = f\left(\frac{\delta w_1(i_1i_2)}{w_1(i_1i_2)}, \frac{\delta w_1(i_2i_3)}{w_1(i_2i_3)}, \frac{\delta w_1(i_1i_3)}{w_1(i_1i_3)}\right) \cdot w_2(i_1i_2i_3)
\]
where $f(u_1,u_2,u_3)$ is a function such that $f(0,0,0)=1$ and that monotonically decreases to zero as $u_i\to -1$, for any $i=1,2,3$.
An example of such $f$ is 
\begin{equation}
\label{eq:trian_weights}
f(u_1,u_2,u_3) = 1-\min\{u_1,u_2,u_3\}\, .    
\end{equation}

Second, when setting the weight of an edge to zero we may disconnect the underlying graph and create an all-zero column and row in the Hodge Laplacian. This gives rise to the phenomenon that we call ``homological pollution'', which we will discuss in detail in the next subsection.

\subsection{Homological pollution: inherited almost disconnectedness}
\label{subsec:connetedness}

As the dimension of Hodge homology $\beta_1$ corresponds to the number of zero eigenvalues in $\bar L_1$,  the intuition suggests that if $\bar L_1$ has some eigenvalue that is close to zero, then the simplicial complex is ``close to'' having at least one more $1$-dimensional hole. There are a number of problems with this intuitive consideration.

By \Cref{thm:inherit} for $k=0$, $\sigma_+(\bar L_1)$ inherits $\sigma_+(\bar L_0)$. Hence, if the weights in $W_1$ vary continuously so that a positive eigenvalue in $\sigma_+(\bar L_0)$ approaches $0$, the same happens to $\sigma_+(\bar L_1)$. Assuming the initial graph $\mc G_{\mc K}$ is connected, an eigenvalue that approaches zero in $\sigma(\bar L_0)$ would 
imply that $\mc G_{{\mc K}}$ is approaching disconnectedness. This leads to a sort of  \emph{pollution} of the kernel of $\bar L_1$ as an almost-zero eigenvalue which corresponds to an ``almost'' $0$-dimensional hole (disconnected component) from $\bar L_0$ is inherited into the spectrum of $\bar L_1$, but this small eigenvalue of $\bar L_1$  does not correspond to the creation of a new $1$-dimensional hole in the reduced complex. 

To better explain the problem of homological pollution, let us consider the following illustrative example. 
\begin{example}\label{ex:connect}
Consider the simplicial complex of order 2  depicted in  \Cref{fig:cand_con_a}. In this example we have  $\mc V_0 = \{1,\dots,7\}$, $\mc V_1=\{[1,2],[1,3], [2,3], [2,4], [3,5], [4,5],$ $ [4,6], [5,6], [5,7], [6,7]\}$ and $\mc V_2 = \{ [1,2,3], [4, 5, 6], [5,6,7] \}$, all with  weight equal to one: $w_k\equiv 1$ for $k=0,1,2$. The only existing $1$-dimensional hole is shown in red and thus the corresponding Hodge homology is $\beta = 1$. Now, consider perturbing the weight of edges $[2,4]$ and $[3,5]$ by setting their weights to $\eps>0$  \Cref{fig:cand_con_b}. For small $\eps$, this perturbation implies that the smallest nonzero eigenvalue $\mu_2$ in $\sigma_+(\bar L_0)$  is scaled by $\eps$. As $\sigma_+(\bar L_0)\subseteq \sigma_+(\bar L_1)$, we have that  
$\dim \ker \bar L_1 = 1$ and $\sigma_+(\bar L_1)$ has an arbitrary small eigenvalue, approaching $0$ with $\eps \to 0$. 
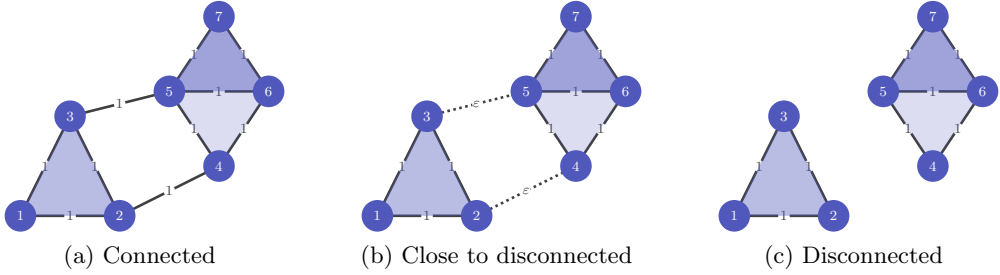
\begin{figure}[t]
    \centering
    { 

\subfloat[Connected]{\label{fig:cand_con_a}
\scalebox{0.66}{\begin{tikzpicture}
    \fill [opacity=0.4, liberty] (0, 0) -- (2, 0) -- (1, 2) -- cycle;
    \fill [opacity=0.2, liberty] (4, 1) -- (3, 2.5) -- (5, 2.5) -- cycle;
    \fill [opacity=0.55, liberty] (4, 4) -- (5, 2.5) -- (3, 2.5) -- cycle;
  \Vertex[x=0, y=0, label=1, style={color=liberty}, fontcolor=white,]{v1}
 \Vertex[x=2,y=0, label=2, style={color=liberty}, fontcolor=white,]{v2}
 \Vertex[x=1, y=2, label=3, style={color=liberty}, fontcolor=white,]{v3}
 \Vertex[x=4, y=1, label=4, style={color=liberty}, fontcolor=white,]{v4}
 \Vertex[x=3, y=2.5, label=5, style={color=liberty}, fontcolor=white,]{v5}
 \Vertex[x=5, y=2.5, label=6, style={color=liberty}, fontcolor=white,]{v6}
 \Vertex[x=4, y=4, label=7, style={color=liberty}, fontcolor=white,]{v7}
 \Edge[Math, label=1](v1)(v2)
 \Edge[Math, label=1](v1)(v3)
 \Edge[Math, label=1](v2)(v3)
 \Edge[Math, label=1](v2)(v4)
 \Edge[Math, label=1](v3)(v5)
 \Edge[Math, label=1](v4)(v5)
 \Edge[Math, label=1](v4)(v6)
 \Edge[Math, label=1](v5)(v6)
 \Edge[Math, label=1](v5)(v7)
 \Edge[Math, label=1](v6)(v7)
\end{tikzpicture}}} $\qquad$
\subfloat[Close to disconnected]{\label{fig:cand_con_b}
\scalebox{0.66}{\begin{tikzpicture}
       \fill [opacity=0.4, liberty] (0, 0) -- (2, 0) -- (1, 2) -- cycle;
    \fill [opacity=0.2, liberty] (4, 1) -- (3, 2.5) -- (5, 2.5) -- cycle;
    \fill [opacity=0.55, liberty] (4, 4) -- (5, 2.5) -- (3, 2.5) -- cycle;
  \Vertex[x=0, y=0, label=1, style={color=liberty}, fontcolor=white,]{v1}
 \Vertex[x=2,y=0, label=2, style={color=liberty}, fontcolor=white,]{v2}
 \Vertex[x=1, y=2, label=3, style={color=liberty}, fontcolor=white,]{v3}
 \Vertex[x=4, y=1, label=4, style={color=liberty}, fontcolor=white,]{v4}
 \Vertex[x=3, y=2.5, label=5, style={color=liberty}, fontcolor=white,]{v5}
 \Vertex[x=5, y=2.5, label=6, style={color=liberty}, fontcolor=white,]{v6}
 \Vertex[x=4, y=4, label=7, style={color=liberty}, fontcolor=white,]{v7}
 \Edge[Math, label=1](v1)(v2)
 \Edge[Math, label=1](v1)(v3)
 \Edge[Math, label=1](v2)(v3)
 \Edge[Math, label=\varepsilon, style={dotted}](v2)(v4)
 \Edge[Math, label=\varepsilon, style={dotted}](v3)(v5)
 \Edge[Math, label=1](v4)(v5)
 \Edge[Math, label=1](v4)(v6)
 \Edge[Math, label=1](v5)(v6)
 \Edge[Math, label=1](v5)(v7)
 \Edge[Math, label=1](v6)(v7)
\end{tikzpicture}}} $\qquad$
\subfloat[Disconnected]{\label{fig:cand_con_c}
\scalebox{0.66}{\begin{tikzpicture}
  \fill [opacity=0.4, liberty] (0, 0) -- (2, 0) -- (1, 2) -- cycle;
    \fill [opacity=0.2, liberty] (4, 1) -- (3, 2.5) -- (5, 2.5) -- cycle;
    \fill [opacity=0.55, liberty] (4, 4) -- (5, 2.5) -- (3, 2.5) -- cycle;
  \Vertex[x=0, y=0, label=1, style={color=liberty}, fontcolor=white,]{v1}
 \Vertex[x=2,y=0, label=2, style={color=liberty}, fontcolor=white,]{v2}
 \Vertex[x=1, y=2, label=3, style={color=liberty}, fontcolor=white,]{v3}
 \Vertex[x=4, y=1, label=4, style={color=liberty}, fontcolor=white,]{v4}
 \Vertex[x=3, y=2.5, label=5, style={color=liberty}, fontcolor=white,]{v5}
 \Vertex[x=5, y=2.5, label=6, style={color=liberty}, fontcolor=white,]{v6}
 \Vertex[x=4, y=4, label=7, style={color=liberty}, fontcolor=white,]{v7}
 \Edge[Math, label=1](v1)(v2)
 \Edge[Math, label=1](v1)(v3)
 \Edge[Math, label=1](v2)(v3)
 \Edge[Math, label=1](v4)(v5)
 \Edge[Math, label=1](v4)(v6)
 \Edge[Math, label=1](v5)(v6)
 \Edge[Math, label=1](v5)(v7)
 \Edge[Math, label=1](v6)(v7)
\end{tikzpicture}}}

    }
    \caption{
    Example of the homological pollution, \Cref{ex:connect}, for the simplicial complex $\mc K$ on $7$ vertices; the existing hole is $[2,3,4,5]$ (left and center pane), all $3$ cliques are included in the simplicial complex and shown in blue. The left pane demonstrates the initial setup with $1$ hole; the center pane retains the hole exhibiting spectral pollution; the continuous transition to the eliminated edges with $\beta_1 = 0$ (no holes) is shown on the right pane. 
    }
    \label{fig:candy_con}
\end{figure}
At the same time, when $\eps\to 0$, the reduced complex obtained by removing the zero edges as in \Cref{fig:cand_con_c} does not have any $1$-dimensional hole, i.e.\ $\beta_1 =0$. Thus, in this case, the presence of a very small eigenvalue $\mu_2 \in \sigma_+ (\bar L_1)$ does not imply that the simplicial complex is close to a simplicial complex with a larger dimension of the Hodge homology.  
\end{example}

To mitigate the phenomenon of homological pollution, in our spectral-based functional for  \Cref{problem:1} we include a term that penalizes the spectrum of $\bar L_0$ from approaching zero. To this end, we observe below that a careful choice of the vertex weights is required. 

The  smallest non-zero eigenvalue of the Laplacian $\mu_2  \in \sigma(\bar L_0)$ is directly related to the connectedness of the graph ${\mc G}_{\mc K}$. This relation is well-known and dates back to the pioneering work of Fiedler \cite{fiedler1989laplacian}.  In particular, as  $\mu_2$ is a function of  node and edge weights,  the following generalized version of the  Cheeger inequality holds~(see~e.g. \cite{tudisco2016})
\begin{equation}\label{eq:cheeger}
    \frac 1 2 \mu_2\leq h(\mc G_{\mc K}) \leq  \left( 2 \, \mu_2\,  \max_{i\in \mc V_0}\frac{\mathrm{deg}(i)}{w_0(i)} \right)^{1/2}, 
\end{equation}
where $h({\mc G}_{\mc K})$ is the Cheeger constant of the graph ${\mc G}_{\mc K}$, defined as 
\[
h({\mc G}_{\mc K})  = \min_{S \subset \mc V_0} \frac{  w_1(S, \mc V_0 \backslash S)    }{ \min \{w_0(S), w_0( \mc V_0 \backslash S )\}     } \, ,
\]
with 
$w_1(S, \mc V_0 \backslash S) = \sum_{ij\in \mc V_1 : i\in S, j\notin S}w_1(ij)$,  $\deg(i) = \sum_{j:ij\in \mc V_1}w_1(ij)$, and $w_0(S)=\sum_{i\in S}w_0(i)$. 

We immediately see from (\ref{eq:cheeger}) that when the graph ${\mc G}_{\mc K}$ is disconnected, then $h({\mc G}_{\mc K})=0$ and $\mu_2=0$ as well. Vice-versa, if $\mu_2$ goes to zero, then $h({\mc G}_{\mc K})$ decreases to zero too. 
While this happens independently of  $w_0$ and $w_1$, if $w_0$ is a function of $w_1$ then it might fail to capture the presence of edges whose weight is decreasing and is about to disconnect the graph.

To see this, consider the example choice  $w_0(i) = \deg(i)$, the degree of node $i$ in ${\mc G}_{\mc K}$. Note that this is  a very common choice in the graph  literature, with several useful properties, including the fact that no other graph-dependent constant appears in the Cheeger inequality (\ref{eq:cheeger}) other than $\mu_2$. For this weight choice, consider the case of a leaf node, a node $i\in \mc V_0$ that has only one edge $ij_0\in \mc V_1$ connecting $i$ to the rest of the (connected) graph ${\mc G}_{\mc K}$ via the node $j_0$. 
If we set $w_1(ij_0)=\varepsilon$ and we let $\varepsilon$ decrease to zero, the graph ${\mc G}_{\mc K}$ is approaching disconnectedness and we would expect $h({\mc G}_{\mc K})$ and $\mu_2$ to decrease as well. However, one easily verifies that both $\mu_2$ and $h({\mc G}_{\mc K})$ are constant with respect to $\varepsilon$ in this case, as long as $\varepsilon\neq 0$.

In order to avoid such a discontinuity, in our weight perturbation strategy for the simplex $\mc K$, if $w_0$ is a function of $w_1$, we perturb it by a constant shift. Precisely, if $w_0$ is the initial vertex weight of $\mc K$, we set $\tilde w_0(i) = w_0(i) + \varrho$, with $\varrho >0$. So, for example, if $w_0=\deg$ and the new edge weight function $\tilde w_1(e) = w_1(e)+\delta w_1(e)$ is formed after the addition of $\delta W_1$, we set $\tilde w_0(i) = \sum_{j} \left[ w_1(ij)+\delta w_1(ij)\right]+\varrho$.

\section{Spectral functional for $1$-st homological stability}
\label{subsec:functional}

We are now in the position to formulate our proposed spectral-based functional, whose minimization leads to the desired small perturbation that changes the first homology of $\mc K$. 
In the notation of \Cref{problem:1}, we are interested in the smallest perturbation $\eps$ and the corresponding modification $\delta W_1 \in \Omega(\eps) \cap \Pi(W_1)$ that increases $\beta_1$. 

As $\|\delta W_1\|=\eps$, for convenience we indicate $\delta W_1 = \eps E$ with $\| E \|=1$\, so $E \in \Omega(1) \cap \Pi_\eps(W_1)$, where  $\Pi_\eps(W_1) = \{W \mid \eps W \in \Pi(W_1)\}$. For the sake of simplicity, we will omit the dependencies and write $\Omega$ and $\Pi_\eps$, when there is no danger of ambiguity. 
Finally, let us denote by $\beta_1(\eps, E)$ the first Betti number corresponding to the simplicial complex perturbed via the edge modification $\eps E$. With this notation, we can reformulate \Cref{problem:1} as follows:

\begin{problem}\label{prob:main}
  Find the smallest $\eps > 0$,
  such that there exists an admissible perturbation $E \in \Omega \cap \Pi_\eps$ with
  an increased number of holes, i.e. 
  \begin{equation} \label{eq:min_1}
  \min\big\{\eps>0 \, :\, \beta_1(\eps, E)\geq \beta_1 + 1 \text{ for some } E \in \Omega \cap \Pi_\eps \big\}
  \end{equation}
  where $\beta_1=\beta_1(0,\cdot)$ is the first Betti number of the original simplicial complex.
  \end{problem}

  In order to approach \Cref{prob:main}, we introduce a target functional $F(\eps,E)$, based on the spectrum of the $1$-Laplacian $\bar L_1(\eps, E)$ and the $0$-Laplacian $\bar L_0(\eps, E)$, where the dependence on $\eps$ and $E$ is to emphasize the corresponding weight perturbation is of the form  $W_1 \mapsto W_1 + \eps E$. 
  
  Our goal is to move a positive entry from $\sigma_+( \bar L_1(\eps, E))$ to the kernel. At the same time,  assuming the initial graph $\mc G_{\mc K}$ is connected, one should avoid the inherited almost disconnectedness with small positive entries of $\sigma_+(\bar L_0 (\eps, E))$. As,  by  \Cref{thm:inherit}  for $k=0$, $\sigma_+(\bar L_0(\eps, E)) = \sigma_+ (\bar L_1^{down}(\eps, E))$,  the only eigenvalue of $\bar L_1(\eps, E)$ that can be continuously driven to $0$ comes from $\bar L_1^{up}(\eps, E)$.  
  For this reason, let us denote the \emph{first non-zero eigenvalue} of the up-Laplacian $\bar L_1^{up}(\eps, E)$ by $\lambda_+ (\eps, E)$. The proposed target functional is defined as:
  \begin{equation}
    \label{eq:functional}
    F(\eps, E) = \frac{ \lambda_+(\eps, E)^2}{ 2 } + \frac{\alpha}{2}  \max \left(0, 1 - \frac{\mu_2(\eps, E)}{\mu} \right)^2
  \end{equation}
where $\alpha$ and $\mu$ are positive parameters, and $\mu_2(\eps, E)$ is the first nonzero eigenvalue of $\bar L_0(\eps, E)$. As recalled in \Cref{subsec:connetedness},   $\mu_2(\eps, E)$ is an algebraic measure of the connectedness of the perturbed graph, thus the whole second term in (\ref{eq:functional}) ``activates'' when such algebraic connectedness falls below the threshold $\mu$.

By design, $F(\eps, E)$ is non-negative and is equal to  $0$ iff $\lambda_+(\eps, E)$ reaches $0$, increasing the  dimension of $\mc H_1$. Using this functional, we recast the \Cref{prob:main} as
\begin{equation}\label{eq:min_2}
  \min \left\{ \eps > 0 :  F(\eps, E) = 0 \text{ for some } E \in \Omega_\eps   \right\}
\end{equation}

\begin{remark}
  When $\mc G_{\mc K}$ is connected, $\dim \ker \bar L_0 = 1$ and, by the \Cref{thm:inherit}, $\dim \ker \bar L_1^{up} = \dim \ker \bar L_1 + ( n - \dim \ker \bar L_0)=n+\beta_1-1$, so the first nonzero eigenvalue of $\bar L_1^{up}$ is the ($n+\beta_1$)-th.
  While $(n+\beta_1)$ can be a large number in practice, we will discuss in \Cref{sec:computational_cost} an efficient method that allows us to  compute $\lambda_+(\eps, E)$ without computing  any of the previous $(n+\beta_1-1)$ eigenvalues. 
\end{remark}

\section{A two-level optimization procedure}
\label{sec:2lev}
  
We  approach problem (\ref{eq:min_2})  by the  means of a two-level iterative method, which is based on the successive 
minimization of the target functional $F(\eps, E)$ and a subsequent tuning of the parameter $\eps$. More precisely, we propose the following two-level scheme.  
  \begin{enumerate}[leftmargin=*]
	\item \textbf{Inner level:} for fixed $\eps > 0$,  solve the minimization problem
  \[
   E(\eps) = \underset{E \in \Omega \cap \Pi_\eps}{\mathrm{arg \, min}} \, F(\eps, E)  
  \]
 by a constrained gradient flow which we formulate below, where we denote the computed minimizer by $E(\eps)$.
    \item \textbf{Outer level:} given the function $\eps \mapsto E(\eps)$, we wish to solve the equation:
    \begin{equation} \label{eq:Fkeps}
      F(\eps, E(\eps)) = 0
    \end{equation}
    and our goal is to compute the smallest solution $\eps^* > 0$ of (\ref{eq:Fkeps}).
  \end{enumerate}
 A similar procedure was used in the context of graph spectral nearness in \cite{andreotti19,andreotti21}
and in other matrix nearness problems \cite{guglub22}.

We approach the inner level   by means of the constrained gradient system 
\begin{equation}
  \label{eq:traj_transition}
  \dot{E}(t)=-\mathbb{P}_{\Omega \cap \Pi_\eps} {G}(\eps, E(t)), \qquad G(\eps, E)=\nabla_{E} F(\eps, E)
\end{equation}
where $\mathbb{P}_{\Omega \cap \Pi_\eps}$ is an orthogonal projector (wrt to Frobenius inner product) 
onto the admissible set $\Omega \cap \Pi_\eps$ (where $\eps$ is fixed). 
Since the system integrates the anti-gradient, a minimizer  (at least local) of 
$F(\eps, E)$ is obtained as $t \to \infty$. 

We devote the next two \Cref{subsec:freegrad} and \Cref{subsec:constrained} to computing the projected gradient in (\ref{eq:traj_transition}). 
The idea is to express the derivative of $F$ in terms of the derivative of the perturbation and this way identifying 
the free gradient, and then projecting it onto the constraints obtaining the constrained gradient. 
To this end, we first compute the free gradient and then we discuss how to deal with the projection onto the admissible set. 

By construction, the resulting algorithm  converges to a (local) minimum of $F(\eps, E)$. Although global convergence  to the global optimum is not guaranteed, in a few simple experiments we  observe the method reaches the expected global solution.

\subsection{The free gradient}
\label{subsec:freegrad}
We compute here the free gradient of $F$ with respect to $E$, given a fixed $\eps$. In order to proceed, we need a few preliminary results.

The following is a standard perturbation result for
eigenvalues; see  e.g. \cite{HornJohnson90}, 
where we denote by $\langle X,Y \rangle =\sum_{i,j} x_{ij}y_{ij} = {\rm Tr}(X^\top Y)$ 
the inner product in ${\mathbb R}^{n\times n}$ that induces the Frobenius norm $\| X \| = \langle X,X \rangle^{1/2}$.

\begin{theorem}[Derivative of simple eigenvalues]\label{lem:eigderiv} 
Consider a continuously differentiable path of square symmetric matrices $A(t)$ for $t$ in an open interval $\mc I$. Let $\lambda(t)$, $t\in \mc I$, be a continuous path of simple eigenvalues of $A(t)$. Let $\vec x(t)$ be the eigenvector associated to the eigenvalue $\lambda(t)$ and assume
$\| \vec x(t) \| = 1$ for all $t$. 
Then $\lambda$ is continuously differentiable on $\mc I$ with the derivative (denoted by a dot) 
\begin{equation}
\dot{\lambda} = \vec x^\top \dot{A} \vec x = \langle \vec  x \vec x^\top, \dot A \rangle\,.
\end{equation}
Moreover, ``continuously differentiable'' can be replaced with ``analytic'' in the assumption and the conclusion.
\end{theorem}

Let us denote the perturbed weight matrix by $\tilde W_1 (t) = W_1 + \eps E(t)$, and the corresponding $\tilde W_0 (t) = W_0(\tilde W_1(t))$ and $ \tilde W_2 (t) = W_2 (\tilde W_1(t) )$, defined accordingly as discussed in \Cref{sec:nearest_complex}. From now on
we omit the time dependence for the perturbed matrices to simplify the notation.  
Since $\tilde W_0$, $\tilde W_1$ and $\tilde W_2$ are necessarily diagonal, by the chain rule we have $\obullet{\tilde{W}}_i (t) = \eps \diag \left( J_1^i  \dot E \vec 1  \right)$, where $\vec 1$ is the vector of all ones, $\diag(\vec v)$ is the diagonal matrix with diagonal entries the vector $\vec v$, and  $J_1^i$ is the Jacobian matrix of the $i$-th weight matrix with respect to $\tilde W_1$, which for any $u_1\in \mathcal V_1$ and $u_2 \in \mathcal V_i$, has entries 
\(
    [J_1^i]_{u_1,u_2}=\frac{\partial }{\partial \tilde{w}_1{(u_1)}}\tilde{w}_i{(u_2)}\, .
\)

Next, in the following two lemmas, we express the time derivative of the Laplacian $\bar L_0$ and $\bar L_1^{up}$ as functions of $E(t)$. The proofs of these results are  straightforward and omitted for brevity. In what follows, $\Sym[A]$ denotes the symmetric part of the matrix $A$, namely $\Sym[A] = (A+A^\top)/2$.

\begin{lemma}[Derivative of $\bar L_0$]
  \label{lem:eigderL0} 
  For the simplicial complex $\mc K$ with the initial edges' weight matrix $W_1$ and  fixed perturbation norm $\eps$, let $E(t)$ be a smooth path and $\tilde W_0, \tilde W_1, \tilde W_2$ be corresponding perturbed weight matrices. Then,
  \begin{equation}
    \frac{1}{2\eps} \frac{d}{dt} \bar L_0 (t)  = \tilde W_0^{-1}  B_1 \tilde W_1 \dot E  B_1^\top \tilde W_0^{-1}-\Sym\left[ \tilde W_0^{-1} \diag \left(  J_1^0 \dot E \vec 1 \right)  \bar L_0 \right]. 
  \end{equation}
\end{lemma}

\begin{lemma}[Derivative of $\bar L_1^{up}$]
  \label{lem:eigderL1}
  For the simplicial complex $\mc K$ with the initial edges' weight matrix $W_1$ and fixed perturbation norm $\eps$,  let $E(t)$ be a smooth path and $\tilde W_0, \tilde W_1, \tilde W_2$ be corresponding perturbed weight matrices.  Then, 
  \begin{equation*}
   \frac{1}{2\eps} \frac{d}{dt} \bar L^{up}_1 (t)  =   - \Sym \left[ \tilde W_1^{-1} B_2 \tilde W_2^2 B_2^\top \tilde W_1^{-1} \dot E \tilde W_1^{-1} \right] 
 + \tilde W_1^{-1} B_2 \tilde W_2 \diag\left(  J_1^0 \dot E \vec 1  \right) B_2^\top \tilde W_1^{-1}
  \end{equation*}
\end{lemma}

Combining \Cref{lem:eigderiv} with  \Cref{lem:eigderL0} and \Cref{lem:eigderL1} we 
obtain the following expression for the free gradient of the functional.
\begin{theorem}[The free gradient of  $F(\eps, E)$]
  \label{thm:fk_grad}
  Assume the initial weight matrices $W_0$, $W_1$ and $W_2$,  as well as the parameters $\eps>0$, $\alpha>0$ and $\mu>0$, are given. 
	Additionally assume that $E(t)$ is a differentiable matrix-valued function such that the first non-zero eigenvalue $\lambda_+(\eps, E)$ of $\bar L_1^{up}(\eps, E)$ and the second smallest eigenvalue $\mu_2(\eps, E)$ of $\bar L_0(\eps, E)$  are simple.  Let $\tilde W_0, \tilde W_1, \tilde W_2$ be corresponding perturbed weight matrices; then: 
 \begin{equation*}
      \begin{aligned}
    \frac{1}{\eps} & \nabla_{E}  F(\eps, E)(t)   = \lambda_+(\eps,E) \bigcdot \\
    & \bigcdot \bigg[ \Sym \left[  - \tilde W_1^{-1} B_2 \tilde W_2^2 B_2^\top \tilde W_1^{-1} \vec x_+ \vec x_+^\top \tilde W_1^{-1}     \right]  \\  &+ \diag \left(  {J_1^2}^\top \vect \left( B_2^\top \tilde W_1^{-1} \vec x_+ \vec x_+^\top \tilde W_1^{-1} B_2 \tilde W_2   \right)   \right) \bigg] - \frac{\alpha}{\mu} \max \left\{ 0, 1- \frac{\mu_2(\eps, E)}{\mu }\right\}  \bigcdot \\
    & \bigcdot \bigg[  B_1^\top \tilde W_0^{-1} \vec y_2 \vec y_2^\top \tilde W_0^{-1} B_1 \tilde W_1    - \diag \left( {J_1^0}^\top  \vect \left( \Sym[ \tilde W_0^{-1} \vec y_2 \vec y_2^\top \bar L_0 ] \right) \right) \bigg]
  \end{aligned} 
 \end{equation*}
  where $\vec x_+$ is a unit eigenvector of $\bar L_1^{up}$ corresponding to $\lambda_+$,  $\vec y_2$ is a unit eigenvector of $\bar L_0$ corresponding to $\mu_2$, and the operator $\vect (X)$ returns the main diagonal of $X$ as a vector. 
\end{theorem}
\begin{proof}
To derive the expression for the gradient $\nabla_E F$, we exploit the chain rule for the time derivative: $\dot \lambda = \langle \frac{d}{dt} A(E(t)), \vec x \vec x^\top  \rangle = \langle \nabla_E \lambda, \dot E \rangle$. Then it is sufficient to apply the cyclic perturbation for the scalar products of \Cref{lem:eigderL0} and \Cref{lem:eigderL1} with $\vec x_+ \vec x_+^\top$ and $\vec y_2 \vec y_2^\top$ respectively. The final transition requires the formula:
 \begin{equation*}
   \left\langle A, \diag(B E \vec 1) \right\rangle = \scal{\diag \left(   B^\top (\vect A) \right) ,\, E}.
 \end{equation*}
\end{proof}
\begin{remark}
  The derivation above assumes the simplicity of both $\mu_2(\eps,E)$ and $\lambda_+(\eps,E)$. This assumption is not restrictive as simplicity for these extremal eigenvalues is a generic property. 
  Indeed we observe simplicity in all our numerical tests.  
\end{remark}

\subsection{The constrained gradient system and its stationary points}
\label{subsec:constrained}
In this section we are deriving from the free gradient determined in  \Cref{thm:fk_grad} the constrained gradient of the considered functional, that is the projected gradient (with respect to the Frobenius inner product) onto the manifold $\Omega \cap \Pi_\eps$,
which consists of perturbations $E$ of unit norm which preserve the structure of $W$.

In order to obtain the constrained gradient system, we need to project the unconstrained gradient given by 
\Cref{thm:fk_grad} onto the feasible set and also to normalize $E$ to preserve its unit norm. 
Using the Karush-Kuhn-Tucker conditions on a time interval where the set of $0$-weight edges remain unchanged, the projection is done via the mapping $\mathbb{P}_{+} G( \eps, E )$, where
\[
\left[\mathbb{P}_+ X \right]_{ij}=\begin{cases} X_{ij}, \quad \left[ W_1 + \eps E \right]_{ij}>0 \\ 0, \;\;\quad \text{otherwise} \end{cases}.
\]
Further, in order to comply with the constraint ${\| E(t) \|^2 =1}$, we must have 
\begin{equation} \label{eq:normconstr}
 0 = \frac12\,\frac d{dt}\| E(t) \|^2= \langle E(t), \dot E(t) \rangle.
\end{equation}
Thus, we obtain the following constrained optimization problem for the admissible direction of the steepest descent

\begin{lemma}[Direction of steepest admissible descent]
\label{lem:opt} 
Let $E,G\in{\mathbb R}^{m \times m}$ with $G$ given by (\ref{eq:traj_transition}), and 
${\|E\|=1}$. On a time interval where the set of $0$-weight edges remains unchanged, 
the gradient system  reads
\begin{equation}
  \label{eq:traj_proj}
  \dot{E}(t)=-\mathbb{P}_+ G( \eps, E(t) )+ \kappa  \mathbb{P}_+E(t), \quad \text{where} \quad 
	\kappa = \frac{\scal{ \eps, G(E(t) ),\, \mathbb{P}_+E(t)}}{  \lVert\mathbb{P}_+E(t)\rVert^2}.
 \end{equation}
\end{lemma}
\begin{proof}
We need to orthogonalize $\dot{E}(t)$ with respect to $E(t)$. To this end, we introduce a linear orthogonality correction, i.e.\ we set  
$\dot E = \mathbb{P}_+ (-G - \kappa E)$, 
{and we determine} $\kappa$ {by imposing the} constraint $\langle E,\dot E \rangle =0$. We then have
\[
0=\langle E,\dot E \rangle = \langle E,\mathbb{P}_+ (-G - \kappa E) \rangle=
- \langle \mathbb{P}_+  E,G \rangle -  \kappa\langle \mathbb{P}_+  E,\mathbb{P}_+  E \rangle,
\]
and the result follows.
\end{proof}	
	
\Cref{eq:traj_proj} suggests that the system goes ``primarily'' in the direction of the antigradient $-G(\eps,E)$, thus the functional is expected to decrease along it.
\begin{lemma}[Monotonicity]
Let $E(t)$ of unit Frobenius norm satisfy the differential equation (\ref{eq:traj_proj}), with $G$ given by (\ref{eq:traj_transition}). Then, 
$F(\eps, E(t))$ decreases monotonically with~$t$.
\end{lemma}
\begin{proof}
We consider first the simpler case where the non-negativity projection does not apply so that $G=G(\eps,E)$ (without $\mathbb{P}_+$). Then
\begin{equation}\label{eq:mono}
    \begin{aligned} 
  \frac{d}{dt} F(\eps, E)(t) & =\scal{  \nabla_{E} F_k(\eps, E), \, \dot{E} }=\scal{ \eps G( \eps, E(t) ) ,\, -G( \eps, E(t) )+\kappa E(t) } \\ 
  & = -\eps \lVert G( \eps, E ) \rVert^2 + \eps\frac{\scal{G( \eps, E ),\, E }}{\scal{E, \, E}} \scal{G( \eps, E ),\, E }  \\
  & = \eps \left(  -\lVert G( \eps, E ) \rVert^2 + \frac{\left| \scal{G( \eps, E ),\, E } \right|^2}{\lVert E \rVert^2} \right) \le 0 
\end{aligned}
\end{equation}
where the final estimate is given by the Cauchy-Bunyakovsky-Schwarz inequality. The derived inequality holds on the time interval without the change in the support of $\mathbb P_+$ (so that no new edges are prohibited by the non-negativity projection).
\end{proof}

\subsection{Free Gradient Transition in the Outer Level}\label{sec:freegrad}

The  optimization problem in the \emph{inner level} is non-convex due to the non-convexity  
of the functional $F(\eps, E)$.
Hence, for a given $\eps$, the computed minimizer $E(\eps)$ may depend on the initial guess $E_0 = E_0(\eps)$.

The effects of the initial choice are particularly important upon the transition 
$ \widehat\eps \to \eps = \widehat\eps + \Delta \eps $ between constrained inner levels: 
given reasonably small $\Delta \eps$, one should expect relatively close optimizers 
$E(\widehat \eps)$ and $E(\eps )$, and, hence, the initial guess 
$E_0(\eps)$ being close to and dependent on $E(\eps)$. 

This choice, which seems very natural, determines however a discontinuity
\[
F(\widehat \eps, E(\widehat \eps)) \ne F(\eps, E(\widehat \eps)),
\]
which   may prevent the expected monotonicity property with respect to $\eps$ in the
(likely unusual case) where $F(\widehat \eps, E(\widehat \eps)) < F(\eps, E(\widehat \eps))$. 
This may happen in particular when $\Delta \eps$ is not taken small; since the choice of $\Delta \eps$
is driven by a Newton-like iteration we are interested in finding a way to prevent this
situation and making the whole iterative method more robust. The goal is that of guaranteeing 
monotonicity of the functional both with respect to time and with respect to $\eps$.

\begin{figure}[t]
  \centering
  { 
  \scalebox{0.9}{\begin{tikzpicture}
  \draw[dashed] (0,0) -- (4.2, 0);
            \draw[dashed] (0,0) -- (0, 4.2);
            \draw[line width=1.5, gray] (2,0) arc (0:90:2.0);
            \draw[line width=1.5, gray] (4,0) arc (0:90:4.0);
            \draw[line width=3, color=liberty, postaction={on each segment={mid arrow=liberty}}] (1.42, 1.42) arc(45:20:2.0);
            \node[anchor = north] at (4, 0){ \tiny{ \( \| \delta W_1 \| = \eps + \Delta \eps \) } };
            \node[anchor = north] at (2, 0){ \tiny{ \( \| \delta W_1 \| = \eps \) } };
            \draw[line width=3, color=burntsienna, postaction={on each segment={mid arrow=burntsienna}}] (1.88, 0.684) arc(-90:-30:2.0);
            \draw[line width=3, color=liberty, postaction={on each segment={mid arrow=liberty}}] (3.61, 1.7) arc(25:60:4);
            \draw[line width=3, color=burntsienna, postaction={on each segment={mid arrow=burntsienna}}] (2, 3.5) arc(160:120:2.0);
            \node[color=liberty, align = center, execute at begin node=\setlength{\baselineskip}{1.5ex}] at (1, 2.7){ \footnotesize constrained flow \\
            \footnotesize  \( \| E(t) \| \equiv 1 \) };
            \node[color=burntsienna, align = center, execute at begin node=\setlength{\baselineskip}{1.5ex}] at (3.5, 3.5){ \footnotesize free flow \\
            \footnotesize  \( \| E(t) \| \uparrow \) };
            \node[circle, black, fill, scale=0.4] at (1.88, 0.684){};
            \node[circle, black, fill, scale=0.4] at (2, 3.5){};
            \node[anchor = east] at (1.88, 0.684){ \tiny \( \| \nabla F \| =0 \) };
            \node[anchor =  east] at (2, 3.4){ \tiny \( \| \nabla F \| =0 \) };
            \draw[color=liberty] (1.5, 2.35) -- (1.75, 1.25);
            \draw[color=liberty] (1.9, 2.6) -- (3., 2.25);
            \draw[color=burntsienna] (2.8, 3.5) -- (2.4, 3.7);
            \draw[color=burntsienna] (3.5, 3.1) -- (3.2, 1.5);
\end{tikzpicture}}
  }

  \caption{
  The scheme of alternating constrained (blue, $\| E(t) \| \equiv 1 $) and free gradient (red) flows. Each stage inherits the final iteration of the previous stage as initial $E_0(\eps_i)$ or $\tilde E_0(\eps_i)$ respectively; constrained gradient is integrated till the stationary point ($\| \nabla F(E) \| = 0 $), free gradient is integrated until $ \| \delta W_1 \| = \eps_i + \Delta \eps$. The scheme alternates until the target functional vanishes ($F(\eps, E) =0 $).
  }
  \label{fig:homotopy}
\end{figure}
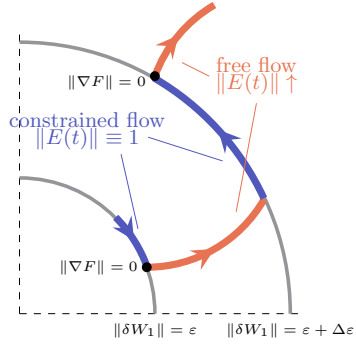

When in the outer iteration we increase $\eps$ from a previous value $\widehat{\eps} < \eps$, we have the problem
of choosing a suitable initial value for the constrained gradient system (\ref{eq:traj_proj}), such that at the
stationary point $E(\widehat \eps)$ we have $F(\widehat \eps, E(\widehat \eps)) < F(\eps, E(\eps))$ (which we assume both
positive, that is on the left of the value $\eps^\star$ which identifies the closest zero of the functional).

In order to maintain monotonicity with respect to time and also with respect to $\eps$, it is convenient to start 
to look at the optimization problem with value $\eps$, with the
initial datum $\delta W_1 = \widehat{\eps} E(\widehat \eps)$ of norm $\widehat \eps < \eps$. 

This means we have essentially to optimize with respect to the inequality constraint $\|\delta W_1 \| \le \eps$,
or equivalently solve the problem (now with inequality constrain on $\| E \|_F$):
\[
   E(\eps) = \underset{ E \in \Omega, \| E \| \le 1}{\mathrm{arg \, min}} \, F(\eps, E)  
\]

The situation changes only slightly from the one discussed above. 
If $\|E\| < 1$, every direction is admissible, and the direction of the steepest descent 
is given by the negative gradient. So we choose the free gradient flow
\begin{equation}\label{ode-E-free}
\dot E = -\mathbb{P}_+ G( \eps, E(t) ) \qquad\text{as long as }\ \| E(t) \| < 1.
\end{equation}
When  $\|E(t)\|=1$, then there are two possible cases. If 
$\langle \mathbb{P}_+ G( \eps, E ), E \rangle \ge 0$, then the solution of (\ref{ode-E-free}) has
\[
\frac d{dt} \|E(t)\|^2 = 2 \,\langle \dot E, E \rangle= -2\, \langle \mathbb{P}_+ G( \eps, E(t) ), E \rangle \le 0,
\]
and hence the solution of (\ref{ode-E-free}) remains of Frobenius norm at most 1. 

Otherwise, if $\langle \mathbb{P}_+ G( \eps, E ), E \rangle < 0$, 
the admissible direction of steepest descent is given by the right-hand side of (\ref{eq:traj_proj}), 
and so we choose this ODE to evolve $E$. 
The situation can be summarized as taking, if $\|E(t)\|=1$,
\begin{equation} \label{ode-E-mu}
\dot E = -\mathbb{P}_+ G( \eps, E ) + \mu E \quad\text{ with }\ \mu=\min\bigl(0,\kappa \bigr)
\end{equation}
with $\kappa = \scal{ G( \eps, E ),\, \mathbb{P}_+E}/  \lVert\mathbb{P}_+E\rVert^2$. 
Along the solutions of (\ref{ode-E-mu}), the functional $F$ decays monotonically, and stationary points of (\ref{ode-E-mu}) (i.e.\ points such that $\dot E=0$) with
$\mathbb{P}_+ G( \eps, E(t) ) \ne 0$ are characterized by
\begin{equation}\label{stat-neg}
\text{$E$ is a {\it negative} real multiple of $\mathbb{P}_+ G( \eps, E(t) )$.}
\end{equation}
If it can be excluded that the gradient $\mathbb{P}_+ G( \eps, E(t) )$ vanishes at an optimizer, 
it can thus be concluded that the optimizer of the problem with inequality constraints is a stationary point of the gradient flow (\ref{eq:traj_proj}) for the problem with equality constraints.

\begin{remark}
  As a result, $F(\eps, E(t))$  monotonically decreases both with respect to time $t$ and to the value of the norm 
  $\eps$, when $\eps \le \eps^\star$. 
\end{remark}

\section{Algorithm  details}\label{sec:algorithm}
In \Cref{algo:complete}  we provide the pseudo-code of the whole bi-level procedure. The initial ``$\alpha$-phase'' is used to choose an appropriate value for the regularization parameter $\alpha$. In order to avoid the case in which the penalizing term on the right-hand side of (\ref{eq:functional}) dominates the loss $F(\varepsilon, E(t))$ in the early stages of the descent flow, we select $\alpha$ by first running such an initial phase, prior to the main alternated constrained/free gradient loop. In this phase,  we fix a small $\eps = \eps_0$  and run the constrained gradient integration for an initial  $\alpha=\alpha_*$. After the computation of a local optimum $E_*$, we then increase $\alpha$ and rerun for the same $\eps_0$ with $E_*$ as starting point. We iterate until no change in $E_*$ is observed or until $\alpha$ reaches an upper bound $\alpha^*$. 

The resulting value of $\alpha$ and $E_*$ are then used in the main loop where we first increase $\eps$ by the chosen step size, we rescale $E_i$ by $0<\eps/(\eps+\Delta \eps)<1$, and then we perform the free gradient integration described in  \Cref{sec:freegrad} until we reach a new point $E_i$ on the unit sphere $\|E_i\|=1$. Then, we perform the inner constrained gradient step by integrating 
\Cref{eq:traj_proj}, iterating the following two-step norm-corrected Euler scheme: 
\begin{equation}
  \label{eq:euler_step}
  \begin{cases}
    E_{i+1/2} = E_i - h_i \left( \mathbb{P}_+ G(E_i, \eps) - \kappa_i \mathbb{P}_+ E_i \right)\, .
    \\
    E_{i+1}=\mathbb P_{\Pi_\eps} E_{i+1/2} / \|\mathbb P_{\Pi_\eps} E_{i+1/2}\| 
  \end{cases}
\end{equation}
where the second step is necessary to numerically guarantee the Euler integration remains in the set of admissible flows since the discretization does not conserve the norm and larger steps $h_i$ may violate the non-negativity of the weights. 
\begin{algorithm}[t]
\caption{Pseudo-code of the complete constrained- and free-gradient flow.}
\label{algo:complete}
\begin{algorithmic}[1]
\Require initial edge perturbation guess $ E_0 $;
initial $ \eps_0 >0$;
$\eps$-stepsize $ \Delta \eps > 0$;
bounds $ \alpha_*, \, \alpha^* $ for the $\alpha$-phase; 
\State $\alpha, E \gets \texttt{AlphaPhase}(E_0, \eps_0,  \alpha_*, \alpha^*)$ \Comment{for details see \Cref{sec:experiments} }
\While{ $ | F(\eps, E) | < 10^{-6} $ }
\State $\eps \gets \eps + \Delta \eps $
\State $ E \gets \frac{\eps}{\eps + \Delta \eps} E $ \Comment{ before the free gradient $\| E \| < 1 $ }
\State $ E_i \gets \texttt{FreeGradientFlow}(E, \Delta \eps, \eps)$  \Comment{
see \Cref{sec:freegrad} }
\State $ E \gets \texttt{ConstrainedGradientFlow}(E, \eps)$  \Comment{see \Cref{sec:algorithm}}
\EndWhile
\end{algorithmic}
\end{algorithm}

In both the free and constrained integration phases, since we aim to obtain the solution at $t \to \infty$ instead of the exact trajectory, we favor larger steps $h_i$ given that the established monotonicity is conserved. Specifically, if $F(\eps, E_{i+1}) < F(\eps, E_i)$, then the step is \emph{accepted} and we set $h_{i+1} = \beta h_i$ with $\beta >1$; otherwise, the step is \emph{rejected} and repeated with  a smaller step $h_i \gets h_i/\beta$.

\begin{remark}
  \label{rem:accelerate}
    The step  acceleration strategy described above, where $\beta h_i$ is immediately increased after one accepted step, may lead to ``oscillations'' between accepted and rejected steps in the event the method would prefer to maintain the current step size $h_i$. To avoid this potential issue, in our experiments we actually increase the step length after two consecutive accepted steps. Alternative step-length selection strategies are also possible, for example, based on Armijo's rule or  non-monotone stabilization techniques \cite{grippo1991class}.
\end{remark}

\subsection{Computational costs} \label{sec:computational_cost}
Each step of either the free or the constrained flows
requires one step of explicit Euler integration along the anti-gradient $-\nabla_E F(\eps, E(t))$. As discussed in  \Cref{sec:2lev}, the  construction of such a gradient requires several sparse and diagonal matrix-vector multiplications  as well as the computation of the smallest nonzero eigenvalue of both $\bar L_1^{up}(\eps, E)$ and $\bar L_0(\eps, E)$. The latter two represent the  major computational requirements of the numerical procedure. Fortunately, as both matrices are of the form $A^\top A$, with $A$ being either of the two boundary or co-boundary operators  $\bar B_2$ and $\bar B_1^\top$,  we have that both the two eigenvalue problems boil down to a problem of the form
\[
\min_{\vec x\bot \ker A}\frac{\|A\vec x\|}{\|\vec x\|}
\]
i.e., the computation of the smallest singular value of the sparse matrix $A$. This problem can be approached by a sparse singular value solver based on a Krylov subspace scheme for the pseudo inverse of $A^\top A$. In practice, we implement the pseudo inversion by solving the corresponding least squares problems
\begin{equation*}
         \min_{\vec x} \| \bar L_1^{up}(\eps, E) \vec x - \vec  b \|, \qquad \min_{\vec x} \| \bar L_0(\eps, E) \vec x - \vec  b \|\, ,
    \end{equation*}
which, in our experiments, we solved using the least square minimal-residual method (LSMR) from \cite{fong2011lsmr}. This approach allows us to use a  preconditioner for the normal equation corresponding to the least square problem. For simplicity, in our tests we used a constant preconditioner computed by means of an incomplete Cholesky factorization of  the initial unperturbed  $\bar L_1^{up}$, or $\bar L_0$. Possibly, much better performance can be achieved with a tailored preconditioner that is efficiently updated throughout the matrix flow. We also note  that the eigenvalue problem for the graph Laplacian $\bar L_0(\eps, E)$ may be alternatively approached by a combinatorial multigrid strategy \cite{spielman2014nearly}. However, designing a suitable preconditioning strategy goes  beyond the scope of this work and will be the subject of future investigations.

\section{Numerical experiments}\label{sec:experiments}
In this section, we provide several synthetic and real-world example applications of the proposed stability algorithm. The code for all the examples is available at \href{https://github.com/COMPiLELab/HOLaGraF}{https://github.com/COMPiLELab/HOLaGraF}. All experiments are run until the global stopping criterion 
$|F(\varepsilon,E(t))|<10^{-6}$ is met. The parameters $\mu$ and $\alpha$ are chosen as follows. Concerning $\mu$, 
we set $\mu = 0.75 \mu_2$, where $\mu_2$ is the smallest nonzero eigenvalue of the initial Laplacian $\bar L_0$. As for $\alpha$, we  run the $\alpha$-phase described in  \Cref{sec:algorithm} with parameters $\eps_0 = 10^{-3}$, $\alpha_*=1$ and $\alpha^*=100$.

\subsection{Illustrative Example}\label{sec:illustrarive_example}

We consider here a small example 
of a simplicial complex $\mc K$ of order 2 
consisting of eight 0-simplices (vertices), twelve 1-simplices (edges),  four 2-simplices $\mc V_2 = \{[1,2,3], [1,2,8], [4,5,6],[5,6,7]\}$   and one corresponding hole $[2,3,4,5]$, hence, $\beta_1 = 1$. By design, the dimensionality of the homology group $\bar{\mc H}_1$ can be increased only by eliminating edges $[1,2]$ or $[5,6]$; for the chosen weight profile  $w_1([1,2]) > w_1([5,6])$, hence, the method should converge to the minimal perturbation norm $\eps = w_1([5,6])$ by eliminating the edge $[5,6]$, \Cref{fig:illustrative_start}.

\begin{figure}[t]
    \centering
    \scalebox{0.7}{\begin{tikzpicture}
            \fill [opacity=0.25,liberty]    (0, 0) -- (2, 0) --  (1, 2) -- cycle;
             \fill [opacity=0.3,liberty]    (0, 0) -- (2, 0) --  (1, -2) -- cycle;
            \fill [opacity=0.5,liberty]    (4, 1) -- (3, 2.5) --  (5, 2.5) -- cycle;
            \fill [opacity=0.65,liberty]    (3, 2.5) -- (5, 2.5) --  (4, 4) -- cycle;
            \Vertex[x=0, y=0, style={color=liberty}, fontcolor=white, size=0.4, label = 6 ]{v1}
            \Vertex[x=2,y=0, style={color=liberty}, fontcolor=white, size=0.4, label = 5 ]{v2}
            \Vertex[x=1, y=2, style={color=liberty}, fontcolor=white, size=0.4, label=4]{v3}
            \Vertex[x=4, y=1, style={color=liberty}, fontcolor=white, size=0.4, label=3]{v4}
            \Vertex[x=3, y=2.5, style={color=liberty}, fontcolor=white, size=0.4, label=2]{v5}
            \Vertex[x=5, y=2.5, style={color=liberty}, fontcolor=white, size=0.4, label=1]{v6}
            \Vertex[x=4, y=4, style={color=liberty}, fontcolor=white, size=0.4, label=8]{v7}
            \Vertex[x=1, y=-2, style={liberty}, fontcolor=white, size = 0.4, label = 7]{v8}
            \Edge[Math, label=1](v1)(v2)
            \Edge[Math, label=0.6, lw=0.6pt](v1)(v3)
            \Edge[Math, label=2.2, lw=2.2pt](v2)(v3)
            \Edge[Math, label=0.7, lw=0.7](v2)(v4)
            \Edge[Math, label=2.5, lw=2.5](v3)(v5)
            \Edge[Math, label=1.8, lw=1.8](v4)(v5)
            \Edge[Math, label=1](v4)(v6)
            \Edge[Math, label=1.5, lw=1.5](v5)(v6)
            \Edge[Math, label=1.75, lw=1.75](v5)(v7)
            \Edge[Math, label=2, lw=2](v6)(v7)
            \Edge[Math, label=0.5,  lw =0.5pt](v1)(v8)
             \Edge[Math, label=3.0,  lw=3.0pt](v2)(v8)
        \end{tikzpicture}}
        \scalebox{0.7}{\begin{tikzpicture}
             \node at (-1, 0.5) {\large $\longrightarrow$};
            \fill [opacity=0.5,liberty]    (4, 1) -- (3, 2.5) --  (5, 2.5) -- cycle;
            \fill [opacity=0.65,liberty]    (3, 2.5) -- (5, 2.5) --  (4, 4) -- cycle;
            \Vertex[x=0, y=0, style={color=liberty}, fontcolor=white, size=0.4, label = 6 ]{v1}
            \Vertex[x=2,y=0, style={color=liberty}, fontcolor=white, size=0.4, label = 5 ]{v2}
            \Vertex[x=1, y=2, style={color=liberty}, fontcolor=white, size=0.4, label=4]{v3}
            \Vertex[x=4, y=1, style={color=liberty}, fontcolor=white, size=0.4, label=3]{v4}
            \Vertex[x=3, y=2.5, style={color=liberty}, fontcolor=white, size=0.4, label=2]{v5}
            \Vertex[x=5, y=2.5, style={color=liberty}, fontcolor=white, size=0.4, label=1]{v6}
            \Vertex[x=4, y=4, style={color=liberty}, fontcolor=white, size=0.4, label=8]{v7}
            \Vertex[x=1, y=-2, style={liberty}, fontcolor=white, size = 0.4, label = 7]{v8}
            \Edge[Math, label=0.6, lw=0.6pt](v1)(v3)
            \Edge[Math, label=2.2, lw=2.2pt](v2)(v3)
            \Edge[Math, label=0.7, lw=0.7](v2)(v4)
            \Edge[Math, label=2.5, lw=2.5](v3)(v5)
            \Edge[Math, label=1.8, lw=1.8](v4)(v5)
            \Edge[Math, label=1](v4)(v6)
            \Edge[Math, label=1.5, lw=1.5](v5)(v6)
            \Edge[Math, label=1.75, lw=1.75](v5)(v7)
            \Edge[Math, label=2, lw=2](v6)(v7)
            \Edge[Math, label=0.5,  lw =0.5pt](v1)(v8)
             \Edge[Math, label=3.0,  lw=3.0pt](v2)(v8)
        \end{tikzpicture}}
    \caption{
    Simplicial complex $\mc K$ on $8$ vertices for the illustrative run (on the left): all 2-simplices from $\mc V_2$ are shown in blue, the weight of each edge $w_1(e_i)$ is given on the figure. On the right: perturbed simplicial complex $\mc K$ through the elimination of the edge $[5,6]$ creating additional hole $[5, 6, 7, 8]$.
    \label{fig:illustrative_start}
    }
\end{figure}
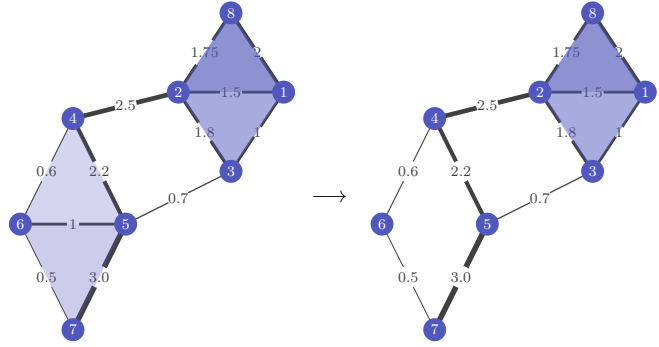

\begin{figure}[t]
    \centering
    \includegraphics[width=1.0\textwidth,clip,trim=8pt 7pt 8pt 5pt]{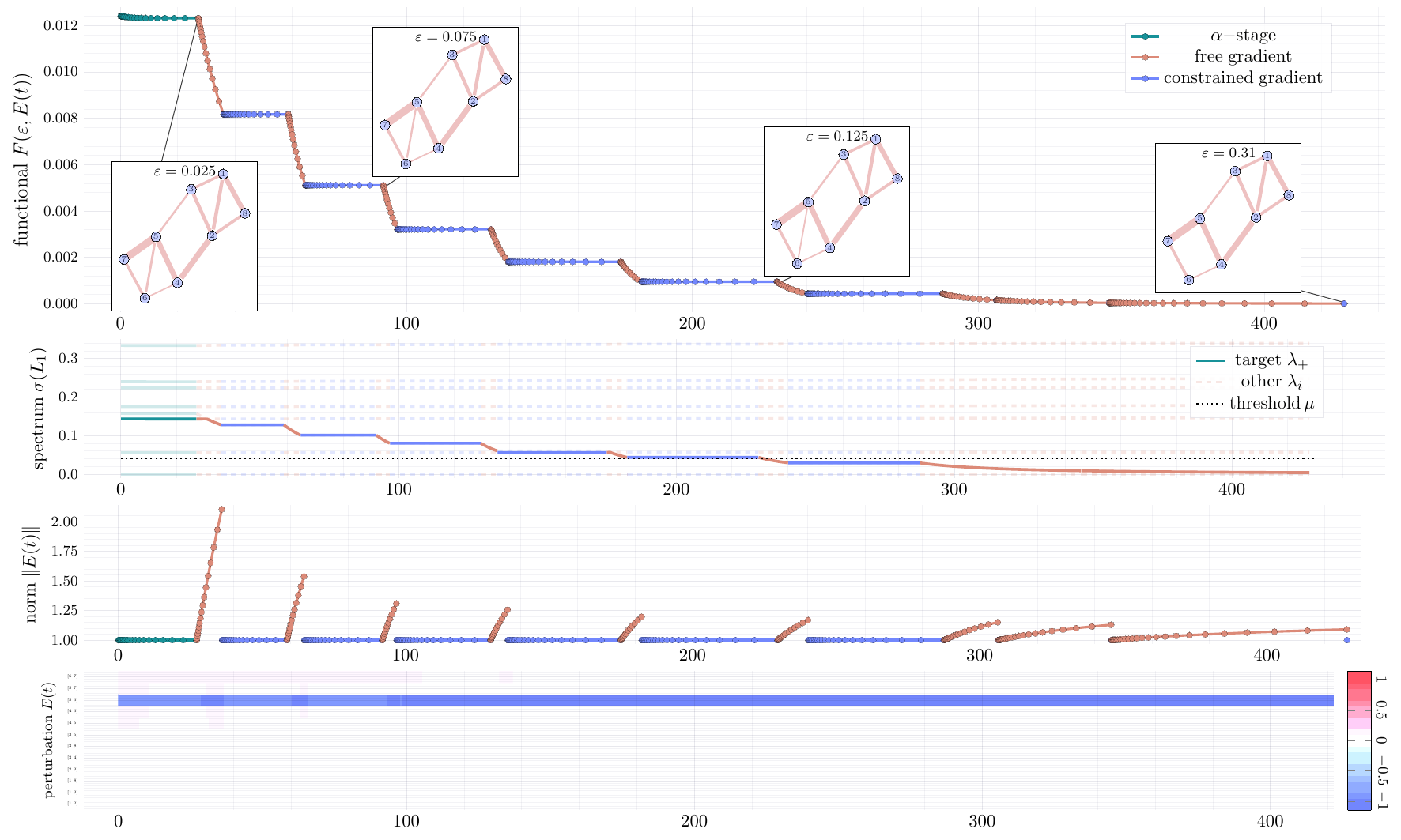}
    
    \caption{Illustrative run of the framework determining the topological stability: the top pane --- the flow of the functional $F(\eps, E(t))$; the second pane --- the flow of $\sigma(\bar L_1)$, $\lambda_+$ is highlighted; third pane --- the change of the perturbation norm $\| E(t) \|$; the bottom pane --- the heatmap of the perturbation profile $E(t)$.
     \label{fig:illustrative}
    }
\end{figure}

The exemplary run of the optimization framework in time is shown on \Cref{fig:illustrative}.
The top panel of \Cref{fig:illustrative} provides the continued flow of the target functional $F(\eps, E(t))$ consisting of the initial $\alpha$-phase (in green) and alternated constrained (in blue) and free gradient (in orange) stages. As stated above, $F(\eps, E(t))$ is strictly monotonic along the flow since the support of $\mathbb P_+$  does not change. Since the initial setup is not pathological with respect to the connectivity, the initial  $\alpha$-phase essentially reduces to a single constrained gradient flow and terminates after one run with $\alpha=\alpha_*$.  The constrained gradient stages are characterized by a slow changing $E(t)$, which is essentially due to the flow performing small adjustments to find the correct rotation on the unit sphere, whereas the free gradient stage quickly decreases the target functional.

The second panel shows the  behaviour of first non-zero eigenvalue $\lambda_+(\eps, E(t))$ (solid line) of $\bar L_1^{up}(\eps, E(t))$ dropping through the ranks of $\sigma(\bar L_1(\eps, E(t)))$ (semi-transparent); similar to the case of the target functional $F(\eps, E(t))$, $\lambda_+(\eps, E(t))$ monotonically decreases. The rest of the eigenvalues exhibit only minor changes, and the rapidly changing $\lambda_+$ successfully passes through the connectivity threshold $\mu$ (dotted line). 

The third and the fourth panels show the evolution of the norm of the perturbation $\| E(t) \|$ and the perturbation $E(t)$ itself, respectively.  The norm $\| E(t) \| $ is conserved during the constrained-gradient and the $\alpha$- stages; these stages correspond to the optimization of the perturbation shape, as shown by the small positive values at the beginning of the bottom panel which eventually vanish. During the free gradient integration the norm $\| E(t) \|$ increases, but the relative change of the norm declines with the growth of $\eps_i$ to avoid jumping over the smallest possible $\eps$. Finally, due to the simplicity of the complex, the  edge we want to eliminate, $56$, dominates the flow from the very beginning (see bottom panel); such a clear pattern persists only in small examples, whereas for large networks the perturbation profile is initially spread out among all the edges.

\subsection{Triangulation Benchmark}

To provide more insight into the computational behavior of the method, we synthesize here an almost planar 
graph dataset. Namely, we assume $N$ uniformly sampled vertices on the unit square with a network built by the Delaunay triangulation;  then, edges are randomly added or erased to obtain the sparsity $\nu$  (so that the graph has $\frac 12  \nu N(N-1)$ edges overall). An order-2 simplicial complex $\mc K=(\mc V_0,\mc V_1,\mc V_2)$  is then formed by letting $\mc V_0$ be the generated vertices, $\mc V_1$ the edges, and $\mc V_2$ every $3$-clique of the graph; edges' weights are sampled uniformly between $1/4$ and $3/4$, namely $w_1(e_i) \sim U [\frac{1}{4}, \frac{3}{4} ]$.

An  example of such triangulation is shown in \Cref{fig:triang}; here $N=8$ and edges $[6, 8]$ and $[2, 7]$ were eliminated to achieve the desired sparsity.

\begin{figure}[t]
    \centering
    \subfloat[Example of Triangulation and Holes]{ \label{fig:triang}
        \scalebox{0.45}{\begin{tikzpicture}
        \Vertex[x=0, y=0, label=1, style={color=liberty}, fontcolor=white,]{v1}
        \Vertex[x=5, y=0, label=2, style={color=liberty}, fontcolor=white,]{v2}
        \Vertex[x=5, y=5, label=3, style={color=liberty}, fontcolor=white,]{v3}
        \Vertex[x=0, y=5, label=4, style={color=liberty}, fontcolor=white,]{v4}
        \Vertex[x=2, y=1.5, label=5, style={color=liberty}, fontcolor=white,]{v5}
        \Vertex[x=1, y=4, label=6, style={color=liberty}, fontcolor=white,]{v6}
        \Vertex[x=4, y=2.5, label=7, style={color=liberty}, fontcolor=white,]{v7}
        \Vertex[x=3, y=3.5, label=8, style={color=liberty}, fontcolor=white,]{v8}
        \Edge(v1)(v2)
        \Edge(v2)(v3)
        \Edge(v3)(v4)
        \Edge(v1)(v4)
        \Edge(v1)(v5)
        \Edge(v2)(v5)
        \Edge(v3)(v7)
        \Edge(v5)(v7)
        \Edge(v7)(v8)
        \Edge(v5)(v8)
        \Edge(v5)(v6)
        \Edge(v3)(v8)
        \Edge(v4)(v6)
        \Edge(v1)(v6)
        \Edge(v3)(v6)

        \fill [opacity=0.3,liberty]   (v1.center) -- (v2.center) -- (v5.center) -- cycle;
        \fill [opacity=0.3,liberty]   (v1.center) -- (v4.center) -- (v6.center) -- cycle;
        \fill [opacity=0.3,liberty]   (v1.center) -- (v5.center) -- (v6.center) -- cycle;
        \fill [opacity=0.3,liberty]   (v3.center) -- (v4.center) -- (v6.center) -- cycle;
        \fill [opacity=0.3,liberty]   (v5.center) -- (v7.center) -- (v8.center) -- cycle;
        \fill [opacity=0.3,liberty]   (v3.center) -- (v7.center) -- (v8.center) -- cycle;
        
        \Vertex[x=0, y=-3.5, style={color=white}]{fake}
    \end{tikzpicture}}
    }
    \subfloat[Time (in seconds)]{\label{fig:triang_time}
       \scalebox{0.4}{
        \includegraphics{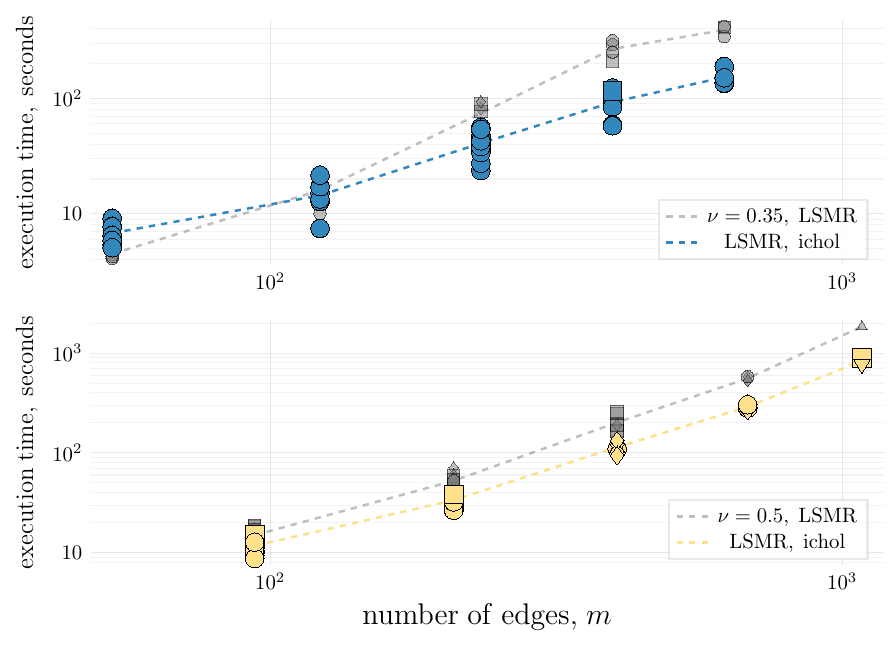}
       }}
    \subfloat[Perturbation norm, $\eps$]{\label{fig:triang_eps}
        \scalebox{0.4}{
        \includegraphics{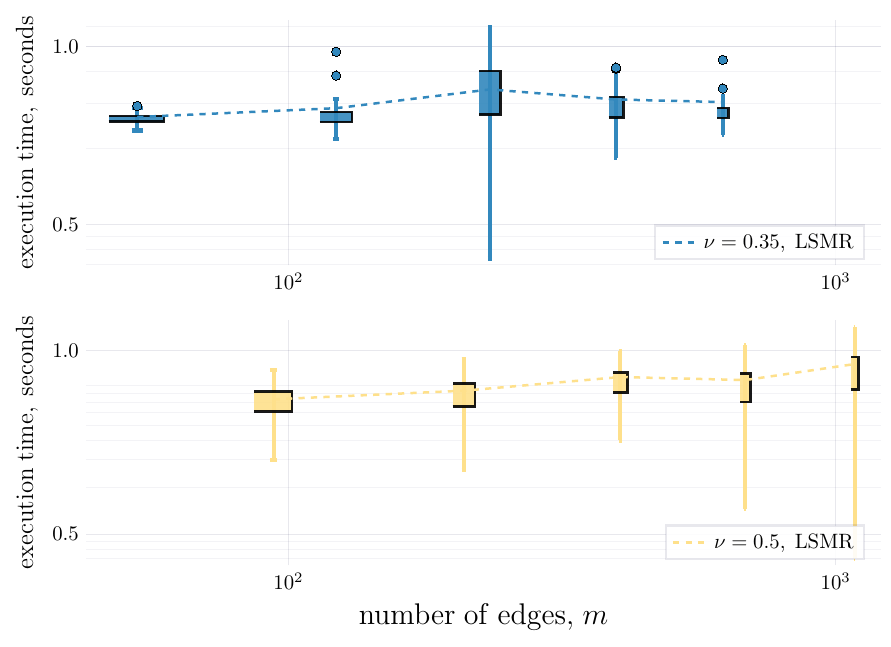}
        }}
        \caption{Benchmarking Results on the Synthetic Triangulation Dataset: varying sparsities $\nu=0.35, \, 0.5$ and $N=16, \, 22, \, 28, \, 34, \, 40$; each network is sampled $10$ times. Shapes correspond to the number of eliminated edges in the final perturbation: $1: \, \ocircle$, $2: \, \square$, $3: \pentagon$, $4: \, \triangle$. For each pair $(\nu, N)$, the un-preconditioned and Cholesky-preconditioned execution times are shown.}
        \vspace{-10pt}
\end{figure}
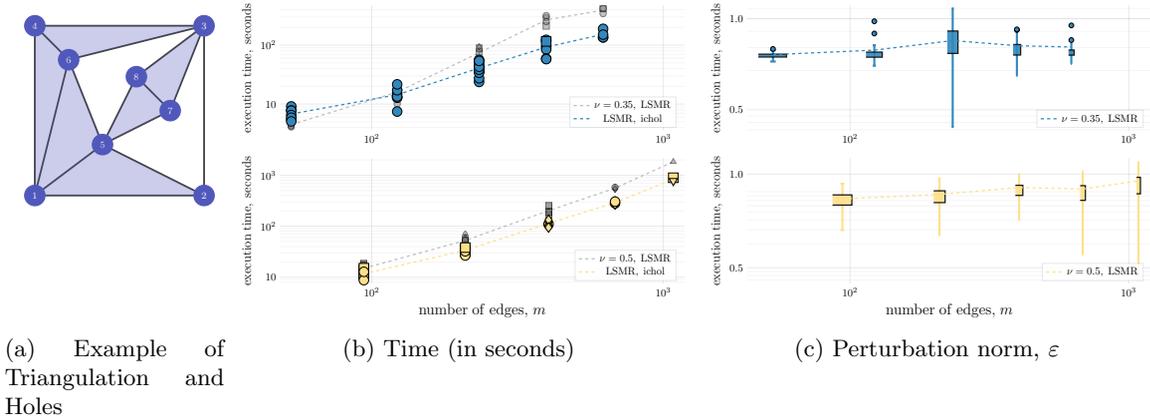

We sample networks with a varying number of vertices $N=10, \, 16, \, 22, \, 28$ and varying sparsity pattern $\nu=0.35, \, 0.5$ which determine the number of edges in the output  as $m = \nu \frac{N(N-1)}{2}$. Due to the highly randomized procedure, topological structures of a sampled graph with a fixed pair of parameters may differ substantially, so $10$ networks with the same $(N, \nu)$ pair are generated. For each network, the working time  (without considering the sampling itself) and the resulted perturbation norm $\eps$, and are reported in  \Cref{fig:triang_time} and \Cref{fig:triang_eps}, respectively. As anticipated in  \Cref{sec:computational_cost}, we show the performance of two implementations of the method, one based on LSMR and one based on LSMR preconditioned by using  the  incomplete Cholesky factorization of the initial matrices. We observe that, 
\begin{itemize}
    \item  the computational cost of the whole procedure  lies between $\mathcal O(m^2)$ and $\mathcal O(m^3)$ 
    \item denser structures, with a higher number of vertices, result in the higher number of edges being eliminated; at the same time, even most dense cases still can exhibit structures requiring the elimination of a single edge, showing that  the flow does not necessarily favor multi-edge optima;
    \item the required perturbation norm $\eps$ is growing with the size of the graph, \Cref{fig:triang_eps}, but not too fast: it is expected that denser networks would require larger $\eps$ to create a new hole; at the same time if the perturbation were to grow drastically with the sparsity $\nu$, it would imply that the method tries to eliminate sufficiently more edges, a behavior that resembles convergence to a sub-optimal perturbation;
    \item preconditioning with a constant incomplete Cholesky multiplier, computed for the initial Laplacians, 
    provides a visible execution time gain for medium and large networks. Since the quality of the preconditioning deteriorates as the flow approaches the minimizer (as a non-zero eigenvalue becomes $0$), it is worth investigating the design of a preconditioner for the up-Laplacian that can be efficiently updated.
\end{itemize}

\subsection{Transportation Networks}

Finally, we provide an application to  real-world examples based on city transportation networks. 
We consider networks for Bologna, Anaheim, Berlin Mitte, and Berlin Tiergarten; each network consists of nodes --- intersections/public transport stops --- connected by edges (roads) and subdivided into zones; for each road the free flow time, length, speed limit are known; moreover, the travel demand for each pair of nodes is provided through the dataset of recorded trips.
All the datasets used here are publicly available at \url{https://github.com/bstabler/TransportationNetworks}; Bologna network is  provided by the Physic Department of the University of Bologna (enriched through the Google Maps API \url{https://developers.google.com/maps}).

The regularity of city maps naturally lacks $3$-cliques, hence forming the simplicial complex based on triangulations as done before frequently leads to trivial outcomes. 
Instead, here we ``lift'' the network to city zones, thus more effectively grouping the nodes in the graph. Specifically:
\begin{enumerate}
    \item we consider the completely connected graph where the nodes are zones in the city/region;
    \item the free flow time between two zones is temporarily assigned  as a weight of each edge: the time is as the shortest path between the zones (by the classic Dijkstra algorithm) 
    on the initial graph;
    \item similarly to what is done in the filtration used for persistent homology,  we filter out excessively distant nodes; additionally, we exclude the longest edges in each triangle in case it is equal to the sum of two other edges (so the triangle is degenerate and the trip by the longest edge is always performed through to others); 
    \item finally, we use the travel demand as an actual weight of the edges in the final network; travel demands are scaled \emph{logarithmically} via the transformation $w_i \mapsto \log_{10} \left(  \frac{w_i}{0.95 \min w_i} \right)$; see the example on the left panel of \Cref{fig:bologna}.
\end{enumerate}
Given the definition of weights in the network, high instability (corresponding to small perturbation norm $\eps$) implies structural phenomena around the ``almost-hole'', where the faster and shorter route is sufficiently less demanded. 

\begin{figure}
    \centering
    \scalebox{1.0}[0.9]{\includegraphics[width=1.0\columnwidth]{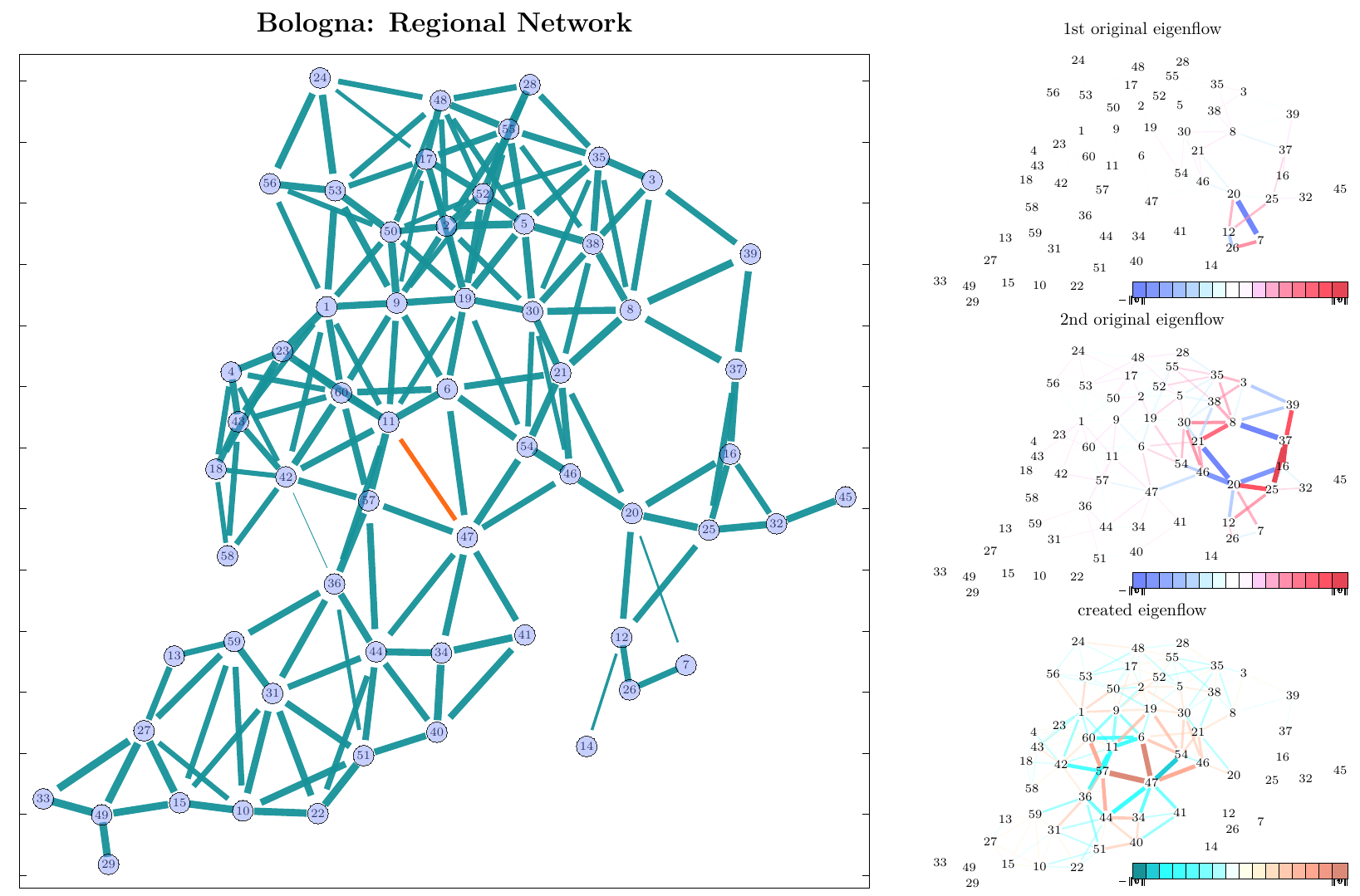}}
    \caption{Example of the Transportation Network for Bologna. Left pane: original zone graph where the width of edges corresponds to the weight, to-be-eliminated edge is colored in red. Right pane: eigenflows, original and created; color and width correspond to the magnitude of entries.}
    \label{fig:bologna}
\end{figure}

\begin{table}[tbhp]
{\footnotesize
  \caption{Topological instability of the transportation networks: filtered zone networks with the corresponding perturbation norm $\eps$ and its percentile among $w_1(\cdot)$ profile. For each simplicial complex the number of nodes, edges and triangles in $\mc V_2(\mc K)$ are provided alongside the initial number of holes $\beta_1$. The results of the algorithm consist of the perturbation norm, $\eps$, computation time, and approximate percentile $p$. }\label{tab:bologna}
\begin{center}

\begin{tabular}{r|ccc|c||ccc|}
\cline{2-8}
                                                         & \multicolumn{3}{c|}{network}                                                                                & \multirow{2}{*}{$\beta_1$} & \multicolumn{3}{c|}{logarithmic weights}                                       \\ 
                                                         & \multicolumn{1}{c}{$n$}                 & \multicolumn{1}{c}{$m$}                  & $\Delta$             &                            & \multicolumn{1}{c}{time}    & \multicolumn{1}{c}{$\eps$}  & $p$      \\ \hline
\multicolumn{1}{|r|}{\multirow{2}{*}{Bologna}}           & \multicolumn{1}{c}{\multirow{2}{*}{60}} & \multicolumn{1}{c}{\multirow{2}{*}{175}} & \multirow{2}{*}{171} & \multirow{2}{*}{2}         & \multicolumn{1}{c|}{$2.43$s} & \multicolumn{1}{c|}{$0.65$}  & $0.003$  \\
\multicolumn{1}{|c|}{}                                   & \multicolumn{1}{c}{}                    & \multicolumn{1}{c}{}                     &                      &                            & \multicolumn{3}{c|}{ \textbf{{[}11, 47{]}} ($4^{th}$ smallest)}                  \\ \hline
\multicolumn{1}{|r|}{\multirow{2}{*}{Anaheim}}           & \multicolumn{1}{c}{\multirow{2}{*}{38}} & \multicolumn{1}{c}{\multirow{2}{*}{159}} & \multirow{2}{*}{221} & \multirow{2}{*}{1}         & \multicolumn{1}{c|}{$5.39$s} & \multicolumn{1}{c|}{$0.57$}  & $0.003$  \\
\multicolumn{1}{|c|}{}                                   & \multicolumn{1}{c}{}                    & \multicolumn{1}{c}{}                     &                      &                            & \multicolumn{3}{c|}{ \textbf{{[}10, 29{]}} ($11^{th}$ smallest)}                 \\ \hline
\multicolumn{1}{|r|}{\multirow{2}{*}{Berlin-Tiergarten}} & \multicolumn{1}{c}{\multirow{2}{*}{26}} & \multicolumn{1}{c}{\multirow{2}{*}{63}}  & \multirow{2}{*}{55}  & \multirow{2}{*}{0}         & \multicolumn{1}{c|}{$2.46$s} & \multicolumn{1}{c|}{$1.18$}  & $0.015$  \\
\multicolumn{1}{|c|}{}                                   & \multicolumn{1}{c}{}                    & \multicolumn{1}{c}{}                     &                      &                            & \multicolumn{3}{c|}{\textbf{{[}6, 16{]}} ($20^{th}$ smallest)}                  \\ \hline
\multicolumn{1}{|r|}{\multirow{2}{*}{Berlin-Mitte}}      & \multicolumn{1}{c}{\multirow{2}{*}{98}} & \multicolumn{1}{c}{\multirow{2}{*}{456}} & \multirow{2}{*}{900} & \multirow{2}{*}{1}         & \multicolumn{1}{c|}{$127$s}  & \multicolumn{1}{c|}{$0.887$} & $0.0016$ \\
\multicolumn{1}{|c|}{}                                   & \multicolumn{1}{c}{}                    & \multicolumn{1}{c}{}                     &                      &                            & \multicolumn{3}{c|}{\textbf{{[}57, 87{]}} ($6^{th}$), \textbf{{[}58, 87{]}}, ($17^{th}$) }               \\ \hline
\end{tabular}

\end{center}
}
\end{table}

In the case of Bologna, \Cref{fig:bologna}, the algorithm eliminates the edge $[11, 47]$ (Casalecchio di Reno -- Pianoro) creating a new hole $6 - 11 - 57 - 47$. We also provide examples of the eigenflows in the kernel of the Hodge Laplacian (original and additional perturbed): original eigenvectors correspond to the circulations around holes $7-26-12-20$ and $8-21-20-16-37$ non-locally spread in the neighborhood~\cite{schaub2020random}. 

The results for four different networks are summarized in the \Cref{tab:bologna}; $p$ mimics the percentile, $\eps/\sum_{e\in \mc V_1}{w_i(e)}$, showing the overall small perturbation norm contextually. 
At the same time, we emphasize that except Bologna (which is influenced by the geographical topology of the land), the algorithm does not choose the smallest weight possible; indeed, given our interpretation of the topological instability, the complex for Berlin-Tiergarten is stable and the transportation network is effectively constructed.

\section{Discussion}

In the current work, we formulate the notion of $k$-th order topological stability of a simplicial complex $\mc K$ as the smallest data perturbation required to create one additional $k$-th order hole in $\mc K$.
By introducing an appropriate weighting and normalization, the stability is reduced to a matrix nearness problem solved by a bi-level optimization procedure. Despite the highly nonconvex landscape, our proposed procedure alternating constrained and free gradient steps yields a monotonic descending scheme. Our experiments show that this approach is generally successful in computing  the minimal perturbation increasing $\beta_1(\eps, E)$, even for potentially difficult cases, as the one proposed in  \Cref{sec:illustrarive_example}.

For simplicity, here we limit our attention to the smallest perturbation that introduces only one new hole. However, a simple modification may be employed to address the case of the introduction of $m$ new holes:  include the sum of $m$ nonzero eigenvalues of $L_1^{up}(\eps,E)$ rather than just the first one  in the spectral functional (\ref{eq:functional}). 
We also remark that, due to the spectral inheritance principle  \ref{thm:inherit}, the proposed framework for $\bar{\mc H}_1$ can be in principle extended to a general $\bar{\mc H}_k$; however, this extension requires nontrivial considerations on the data modification procedure and on the numerical linear algebra tools, as a nontrivial topology of higher-order requires a much denser network.

Different improvements are possible in terms of numerical implementation, including investigating the use of more sophisticated (e.g. implicit) integrators for the gradient flow system (\ref{eq:traj_proj}), which   would additionally require the use  of higher-order derivatives of $\lambda_+(\eps, E)$. Moreover, as already mentioned in \Cref{sec:computational_cost}, the numerical method for the computation of the small singular values would benefit from the use of an efficient preconditioner that can be effectively updated throughout the flow. Investigations in this direction are in progress and will be the subject of future work.

\end{document}